\documentclass[11pt]{amsart}
\usepackage[margin=1in]{geometry}
\usepackage{times}
\usepackage{amsmath}
\usepackage{amssymb}
\usepackage{amsthm}
\usepackage{epsfig}
\usepackage{marvosym}
\usepackage{enumerate}
\usepackage{bbm,bm}
\usepackage{graphicx} 
 \usepackage{tikz}
\usetikzlibrary{arrows,automata}
\usepackage{mathtools}
\usepackage{array}
\usepackage{lipsum}
\usepackage{mathrsfs}
\usetikzlibrary{arrows}
\usepackage[scaled]{helvet}
\usepackage[EULERGREEK]{sansmath}
\usepackage[english]{babel}
\usepackage[utf8]{inputenc}
\usepackage{fancyhdr}
\usepackage[bookmarks=true,
bookmarksnumbered=true, breaklinks=true,
pdfstartview=FitH, hyperfigures=false,
plainpages=false, naturalnames=true,
colorlinks=true,pagebackref=true,
pdfpagelabels]{hyperref}
\usepackage{pgf,tikz}

\newtheorem{thm}{Theorem}

\newtheorem{optimization}{Optimization Problem}

\newtheorem{manualtheoreminner}{Theorem}
\newenvironment{manualtheorem}[1]{%
  \IfBlankTF{#1}
    {\renewcommand{\themanualtheoreminner}{\unskip}}
    {\renewcommand\themanualtheoreminner{#1}}%
  \manualtheoreminner
}{\endmanualtheoreminner}

\theoremstyle{definition}
\newtheorem{deff}[thm]{Definition}

\newtheorem{proposition}[thm]{Proposition}

\newtheorem{lemma}[thm]{Lemma}

\theoremstyle{remark}
\newtheorem{rmk}[thm]{Remark}

\newtheorem{corollary}[thm]{Corollary}

\newcommand{\R}{\mathbb{R}} 
\newcommand{\N}{\mathbb{N}}

\newcommand{\e}{\varepsilon}

\DeclareMathOperator*{\argmax}{arg\,max}

\DeclareMathOperator{\vol}{vol}
\DeclareMathOperator{\dom}{dom}
\DeclareMathOperator{\supp}{supp}
\DeclareMathOperator{\intt}{int}

\DeclareMathOperator{\tr}{tr}
\DeclareMathOperator{\divv}{div}

\newcommand{\Sp}{\mathbb{S}^{n-1}}

\begin{document}

	\author{Steven Hoehner and Michael Roysdon}
		\title[Maximal intersection position for logarithmically concave functions and measures]{On maximal intersection position for logarithmically concave functions and measures}
 \date{\today}

	\subjclass{Primary:  46N10, 52A39, 52A40, 60E05;  Secondary: 28A75} \keywords{convex body, convolution, isotropic, John position, log-concave,  maximal intersection position}	

\thanks{Michael Roysdon is supported by an AMS-Simons Travel Grant}

 \begin{abstract}
A  new  position is introduced and studied for the convolution of log-concave functions, which may be regarded as a functional analogue of the maximum intersection position of convex bodies introduced and studied by Artstein-Avidan and Katzin (2018) and Artstein-Avidan and Putterman (2022). Our main result is a John-type theorem for the maximal intersection position of a pair of log-concave functions, including the corresponding decomposition of the identity. The main result holds under very weak assumptions on the functions; in particular, the functions considered may both have unbounded supports.  As an application of our results, we introduce a John-type position for even $\log$-concave measures. 
\end{abstract}
	
	\maketitle


\section{Introduction} \label{s:intro}


We shall work in $n$-dimensional Euclidean space $\R^n$ ($n\geq 1$) equipped with the standard inner product and Euclidean norm $|\cdot|$. The origin in $\R^n$ is denoted $o$. The interior and boundary of a set $A\subset\R^n$ are denoted by $\intt(A)$ and $\partial A$, respectively.

A \emph{convex body} $K\subset\R^n$ is a convex, compact set with nonempty interior. The $n$-dimensional volume of $K$ is denoted $\vol_n(K)$. Every convex body $K$ in $\R^n$ induces a gauge function $\|\cdot\|_K$ defined by $\|x\|_K=\inf\{\lambda\geq 0:\, x\in\lambda K\}$ for $x\in \R^n$. For more background on convex bodies, we refer the reader to, e.g., the book of Schneider \cite{SchneiderBook}.

The maximum volume ellipsoid contained in a convex body $K$ in $\R^n$, called the \emph{John ellipsoid} of $K$, plays a fundamental role in convex geometry and asymptotic geometric analysis. When the John ellipsoid of $K$ is the Euclidean unit ball $B_2^n$, we say that $K$ is in \emph{John position}. The celebrated John's theorem \cite{John,Ball92,Ball97} asserts that a convex body $K$ in $\R^n$ that contains the Euclidean unit ball $B_2^n$ is in John position if and only if there exists a centered isotropic measure on $\Sp$ supported on the contact points of $K\cap B_2^n$. A finite Borel measure $\mu$ on the Euclidean unit sphere $\Sp$ is \emph{isotropic} if 
\begin{equation}\label{e:isotropicclasical}
\int_{\Sp}u\otimes u\,d\mu(u)=I_n,
\end{equation}
where $x\otimes y$ denotes the linear transformation on $\R^n$ defined by $(x\otimes y)(z)=\langle z,y\rangle x$ and $I_n$ is the $n\times n$ identity matrix. The condition \eqref{e:isotropicclasical} is equivalent to the condition that, for every unit vector $\theta$ in $\R^n$, 
\[
\int_{\mathbb{S}^{n-1}} \langle y,\theta \rangle^2\, d\mu(y) = \frac{\mu(\mathbb{S}^{n-1})}{n}. 
\]

\noindent A measure $\mu$ is called \emph{centered} if $\int_{\Sp}u\,d\mu(u)=o$. More precisely, John's theorem may be formulated as follows.

\begin{manualtheorem}{A}[John's Theorem \cite{John,Ball92,Ball97}]
Let $K$ be a convex body in $\R^n$ such that $B_2^n\subset K$. Then the following are equivalent:
\begin{itemize}
    \item[(i)] $B_2^n$ is the maximum volume ellipsoid contained in $K$.

    \item[(ii)]  There exist $c_1,\ldots,c_m>0$ and $u_1,\ldots,u_m\in\partial K\cap\Sp$, with $m\leq n(n+3)/2$, such that $I_n=\sum_{i=1}^m u_i\otimes u_i$ and $\sum_{i=1}^m c_i u_i=o$.
\end{itemize}
\end{manualtheorem}

There are numerous extensions and generalizations of John's theorem in different directions. In one direction, one replaces ellipsoids by \emph{positions}, or affine images,  of another convex body $L$. In another direction, one replaces convex bodies by logarithmically concave (``log-concave") functions on $\R^n$. In the present paper, we pursue both directions and investigate a new  position for arbitrary log-concave functions. The new position may be regarded as a functional analogue of the maximum intersection position of convex bodies (described below). Our main result is a John-type theorem for this new position, including the corresponding decomposition of the identity. To the best of our knowledge, our main result is the first John-type theorem for log-concave functions with \emph{unbounded supports}.

Given a pair of convex bodies $K,L\subset\R^n$, the \emph{maximal volume image} of $K$ inside $L$ is the affine image of $K$ contained in $L$ that has greatest volume. We then say that $K$ is in  \emph{maximal volume position} if $K$ itself is the maximal volume image of $K$ inside $L$. This position has been studied extensively;  see, for example, \cite{BR-2001,GPT,GM-2000,GP-1999,G-L-M-P,GS-2005} and the references therein. In particular, the relation between positions of convex bodies and isotropic measures was studied in  \cite{BR-2001,GPT,GM-2000,GP-1999}, while extensions to related minimization problems were investigated in \cite{BR-2001,BR-2004,G-L-M-P,Lasserre,LYZ-2005}.  For more background on John ellipsoids and positions of convex bodies, we refer the reader to the preceding references,  as well as the monographs of Artstein-Avidan, Giannopoulos and Milman \cite{AGA-Book} and Schneider \cite{SchneiderBook}.

Recently, Artstein-Avidan and Putterman \cite{AP-2022} introduced the maximal intersection position of a pair of convex bodies, which  generalizes the maximal volume position. Two convex bodies $K,L\subset\R^n$ are said to be in \emph{maximal intersection position} if $K$  itself has largest intersection with $L$ among all convex bodies with the same volume as $K$. In the special case when $K$ is a ball of radius $r$ and $L$ is symmetric, it reduces to the \emph{$r$-maximal intersection position} previously introduced and studied by Artstein-Avidan and Katzin in \cite{AK2018} (see also \cite{BH-2022}).

As shown in \cite{AK2018,AP-2022}, these maximal intersection positions also give rise to analogues of John's theorem, including the corresponding decompositions of the identity. In particular, Artstein-Avidan and Putterman \cite{AP-2022} proved the following result. For a convex body $L$ in $\R^n$, let  $\sigma_{\partial L}$ denote the $(n-1)$-dimensional Hausdorff measure restricted to its boundary $\partial L$.  For any $x\in\partial L$, let $n_L(x)$ be the unit normal at $x$, which is defined $\mathcal{H}^{n-1}$-almost everywhere on $\partial L$. 

\begin{manualtheorem}{B}(\cite[Theorem 1.4]{AP-2022})\label{AP-thm}
    Let $K,L\subset\R^n$ be convex bodies that are in maximal intersection position, and suppose  that $\vol_{n-1}(\partial K\cap \partial L)=0$. Then we have 
    \begin{align*}
        \int_{K\cap \partial L}n_L(x)\,d\sigma_{\partial L}(x)&=o,\\
        \int_{K\cap \partial L}x\otimes n_L(x)\,d\sigma_{\partial L}(x) &= C\cdot I_n
    \end{align*}
for some constant $C>0$. 
    The same formulae hold when interchanging the roles of $K$ and $L$.
\end{manualtheorem}

\noindent Please note that the assumption $\vol_{n-1}(\partial K \cap \partial L) = 0$ cannot be omitted from Theorem \ref{AP-thm} (see \cite{AP-2022}).

In the present paper, we introduce and study a  maximal intersection position for pairs of log-concave functions. Our main result, stated as  Theorem~\ref{t:isotropic} below, is an analogue of Theorem \ref{AP-thm} which holds for log-concave functions that satisfy certain conditions on their supports. These conditions are analogous to those imposed on the boundaries of the convex bodies in Theorem \ref{AP-thm}. Interestingly, in a departure from the geometric case of Theorem \ref{AP-thm}, the corresponding decomposition of the identity that we obtain now splits into two parts: a nonsingular term and a singular (boundary) term.  In fact, as our main result shows, the singular term arises only when the functions considered have bounded supports; furthermore, if the functions considered have unbounded supports, then only the nonsingular term arises. In particular, choosing characteristic functions of convex bodies in our main result allows one  to recover the corresponding results from the geometric case in \cite{AK2018,AP-2022}.

The maximal intersection position introduced in the present paper is akin to some of the other positions of log-concave functions  which have been recently investigated, including John ellipsoids  \cite{AGJV,IN-2022}, as well as a John position for pairs of log-concave functions in \cite{IN-2023} (see also \cite{IT-2021,Li-Schutt-Werner} and the references therein). We also note that the minimal perimeter of a log-concave function was studied in \cite{Fang-Zhang}, where a corresponding John-type theorem was also proved.

Before stating the new definition and  our main results, we first recall some preliminaries. A  function is \emph{$\log$-concave} if it is of the form $\exp(-\varphi)$ for some convex function $\varphi \colon \R^n \to (-\infty,\infty]$.  Our principal goal is to study the following problem: given a pair of integrable, sufficiently regular $\log$-concave functions $f$ and $g$, consider the quantity
\begin{equation}\label{e:generaloptimization}
m_{f,g} :=  
\sup\left\{ \int_{\R^n}f(x)g(T^{-1}(x-z))\,dx:\,(T,z) \in \text{SL}_n(\R) \times \R^n\right\}.
\end{equation} 


\noindent The main definition of this paper, the maximal intersection position for a pair of log-concave functions, reads as follows.

\begin{deff}\label{d: maxintpostradr} Suppose that $f,g \colon \R^n \to [0,\infty)$ are a pair of integrable $\log$-concave functions. If $(I_n,o)$ is a solution to \eqref{e:generaloptimization}, then we say that $f$ and $g$ are in \emph{maximal intersection position}. 
    
\end{deff}


We begin by stating two results which are special cases of our main result. Each of these two results highlights an interesting facet of Theorem \ref{t:isotropic}. First, we present a special case in which the decomposition of the identity is in a simplified form. This result will be obtained by simply choosing one of the functions in Theorem \ref{t:isotropic} to be supported on all of $\R^n$, so that the boundary of the support is empty and all boundary terms vanish. In what follows, let $\text{supp}(f)$ denote the support of $f$. 


\begin{thm}\label{t:fulldim} Let $f \colon \R \to [0,\infty)$ be an integrable $\log$-concave function, and let $g= e^{-\varphi}$ be a $\log$-concave function supported on $\R^n$ with $\int_{\R^n} |\nabla g|^2 < \infty$. 
\begin{itemize}
\item[(a)] There exists a pair $(T_0,z_0)\in{\rm SL}_n(\R)\times\R^n$ such that the supremum \eqref{e:generaloptimization} is attained.
\item[(b)] If $f$ and $g$ are in maximal intersection position, then for every $\theta \in \mathbb{S}^{n-1}$,
\[
C_{f,g} I_n = \int_{\R^n} \nabla \varphi(x) \otimes x  f(x) g(x)\, dx,
\]
and $o = \int_{{\rm supp}(f)}f(x)g(x) \varphi(x)\,dx$, where 
\[
C_{f,g} := \frac{1}{n} \int_{{\rm supp}(f)} \langle \nabla \varphi(x), x \rangle f(x)g(x)\, dx.
\]
\end{itemize}
\end{thm}

Another interesting special case of our main result is a John-type theorem for log-concave measures. Furthermore, from this result we will deduce Theorem \ref{AP-thm} as a special case. While John-type positions for $\log$-concave functions have been widely studied in the literature, the same is not true for $\log$-concave measures. A sufficiently regular Borel measure on $\R^n$ is \emph{$\log$-concave} if its density is a $\log$-concave function.   Given a $\log$-concave probability measure $\mu$ on $\R^n$ and nonempty convex sets $K, L \subset \R^n$ and $r >0$, consider
\begin{equation}\label{e:measureopt}
m_{\mu}(r) :=\sup\left\{\mu(K \cap (TL+z)): (T,z) \in {\rm GL}_n(\R) \times \R^n, \,|\det T| = r\right\}.
\end{equation}
This corresponds to choosing $f = 1_K e^{-\psi}$, with $\psi$ supported on the whole of $\R^n$, and $g = 1_L$ in \eqref{e:generaloptimization} when $r = 1$. This leads to the next definition. 

\begin{deff}\label{d:measures} Let $\mu$ be a $\log$-concave probability measure on $\R^n$, $K, L \subset \R^n$ be convex bodies, and let $r>0$.  Then we say that $K$ is in \emph{maximal $\mu$-intersection position of radius $r$ with respect to $L$} if $(I_n,o)$ is a solution to \eqref{e:measureopt}. 
\end{deff}

\begin{thm}\label{t:measure} Let $\mu$ be a $\log$-concave probability measure on $\R^n$, let $K,L \subset \R^n$ be convex bodies and let $r >0$. 

\begin{enumerate}
    \item[(a)] There exists a pair $(T_0,z_0)\in{\rm GL}_n(\R)\times\R^n$ such that the supremum \eqref{e:measureopt} is attained.
\item[(b)] Assume, additionally, that the following hold: 
\begin{itemize}
    \item $\vol_{n-1}(\partial K\cap \partial L) = 0$;
    \item $\vol_{n-1}(\partial K \cap \partial (TL+z)) = 0$ for all but finitely many $(T,z) \in  {\rm SL}_n(\R) \times \R^n$.
\end{itemize} 
If $K$ is in maximal $\mu$-intersection position of radius $r$ with respect to $L$, then for every $\theta \in \mathbb{S}^{n-1}$,
\begin{equation*}
 \begin{split}
 \int_{K \cap \partial L}\langle n_{L}(y),\theta\rangle \langle y, \theta \rangle\, d\mu_{\partial L}(y)
 =\frac{1}{n} \int_{K \cap \partial L}\langle n_{L}(y),y \rangle\, d\mu_{\partial L}(x) .
\end{split}  
\end{equation*}
In particular, the following decomposition of the identity holds:  
\[
\frac{I_n}{n}=\frac{\int_{K \cap \partial L} n_{L}(y) \otimes y\, d\mu_{\partial L}(y) }{\int_{K\cap \partial L}\langle n_{L}(y),y \rangle\, d\mu_{\partial L}(y)}.
\]
Moreover, 
\[
o=  \int_{K \cap \partial L}n_{L}(y) \,d\mu_{\partial L}(y).   
\]
\end{enumerate}
\end{thm}

We  recover Theorem~\ref{AP-thm} by choosing $\mu$ to be the Lebesgue measure on $\R^n$, with the additional requirement specified at the beginning of part (b) of Theorem~\ref{t:measure}.



Both Theorem~\ref{t:fulldim} and Theorem~\ref{t:measure}  follow from  the more general results given in  Theorem \ref{t:exitence} and Theorem \ref{t:isotropic} below. 

 In what follows, we refer to an ellipsoid $\mathcal{E}$ as a nondegenerate linear image of $B_2^n$, i.e., $\mathcal{E} = TB_2^n$ for some $T \in {\rm GL}_n(\R)$.  Let $\mu$ be an even $\log$-concave measure with strictly positive density and $K \subset \R^n$ an origin-symmetric convex body. A standard compactness argument shows that there is an ellipsoid $\mathcal{E}_0 \subset K$ such that 
\[
\mu(\mathcal{E}_0) = \sup\{\mu(\mathcal{E}) \colon \mathcal{E} \subset K,\, \mathcal{E} \text{ is an ellipsoid\}}. 
\]
We call $\mathcal{E}_0$ a \emph{$\mu$-John ellipsoid} of $K$. Finally, we say that $K$ is in the \emph{$\mu$-John position} if $B_2^n$ is a $\mu$-John ellipsoid of $K$. As an application of our results, we establish a John-type position for $\log$-concave measures.

\begin{thm}\label{t:measureJohn} Let $\mu$ be an even $\log$-concave measure whose density is strictly positive, and let $K \subset \R^n$ be an origin-symmetric convex body that is in $\mu$-John position. Then there exists a Borel measure $\nu$ on $\mathbb{S}^{n-1}$, supported in $\partial K \cap \mathbb{S}^{n-1}$, such that 
\[
\frac{I_n}{n} = \int_{\mathbb{S}^{n-1}\cap \partial K} x \otimes x\, d\nu(x). 
\]
\end{thm}


    


Our main results are formally presented  in Section \ref{Main-results}, and their proofs  are  given in Sections \ref{s:existence} and  \ref{s:variationalformulas}.

\section{Maximal intersection position for pairs of log-concave functions}\label{Main-results}

Before stating our main results, let us briefly discuss the notation that will be used throughout the paper. For a convex function $\psi:\R^n\to\R\cup\{\infty\}$, let $\dom(\psi)=\{x\in\R^n:\,\psi(x)<\infty\}$. We say that $\psi$ is \emph{proper} if $\dom(\psi)$ is nonempty, and $\psi$ is \emph{coercive} if $\lim_{|x|\to\infty}\psi(x)=\infty$. The set of points on or above the graph of $\psi$ is a convex set in $\R^{n+1}$ called the \emph{epigraph} of $\psi$. If the epigraph of $\psi$ is closed, then $\psi$ is called \emph{lower semicontinuous}. Let
 \[
\mathcal{C}^n:=\{\psi:\R^n\to\R\cup\{\infty\}:\,\psi\text{ is convex, proper, coercive and lower semicontinuous}\}.
 \]

A canonical way of embedding the class of convex bodies in $\R^n$ (equipped with Minkowski addition) into a functional setting is by considering log-concave functions. A function $f:\R^n\to[0,\infty)$ is \emph{log-concave} if it has the form $f=e^{-\psi}$ for some convex function $\psi:\R^n\to\R\cup\{\infty\}$. Here we use the convention $e^{-\infty}=0$.  Our main results will be stated for the following class of coercive log-concave functions on $\R^n$:
\[
\mathcal{L}^n := \left\{f = e^{-\psi}\colon\, \psi \in \mathcal{C}^n\right\}.
\]

\noindent The \emph{support} of a measurable function $h:\R^n\to\R\cup\{\pm\infty\}$ is the set $\supp(h)= \overline{\{x \in \R^n \colon h(x) \neq 0\}}$. To simplify the notation, throughout the paper we will write  $K_h:=\supp(h)$.


Let $K$ be a convex body in $\R^n$. The \emph{characteristic function} $\mathbbm{1}_K:\R^n\to[0,\infty)$ is defined by $\mathbbm{1}_K(x)=1$ if $x\in K$, and $\mathbbm{1}_K(x)=0$ if $x\not\in K$. Note that $\mathbbm{1}_K\in\mathcal{L}^n$ and $\supp(\mathbbm{1}_K)=K$.

As a special case of Rademacher's theorem (see, e.g., \cite[Theorem 10.8(ii)]{Villani-book}), every convex function $\psi:\R^n\to\R$ is differentiable almost everywhere on the interior of its domain (see also \cite[Theorem 25.5]{RockafellarBook}). Hence, every log-concave function $f:\R^n\to\R$ is differentiable almost everywhere on the interior of its support. 

It is well-known that if $f\in\mathcal{L}^n$, then $f$ is integrable (see, for example, \cite[p. 3840]{cordero-erasquin-klartag}). Thus,  for $f = e^{-\psi} \in \mathcal{L}^n$, we may consider the \emph{total mass} functional 
\[
J(f) = \int_{\R^n}f(x)\, dx = \int_{\R^n} e^{-\psi(x)}\,dx. 
\]
 In particular, $J(\mathbbm{1}_K)=\vol_n(K)$.

For more background on convex  analysis, we refer the reader to, e.g., \cite{RockafellarBook, Rockafellar-Wets}, and for more background on log-concave functions, we refer the reader to, e.g., \cite{Colesanti-inbook}.

\subsection{Existence of the maximal intersection position}

Given a pair of functions $f,g \in \mathcal{L}^n$, consider the functional $P_{f,g}\colon M_n(\R)\times\R^n \to \R$ defined by 
\[
P_{f,g}(T,z) := \int_{\R^n}f(x) g(T^{-1}(x-z))\, dx. 
\]

\noindent For this functional, we consider the following optimization problem.
	
\begin{optimization}\label{q:optimization}
For any fixed pair $f,g \in \mathcal{L}^n$, determine if
\begin{equation}\label{e:opt1}
\sup\{P_{f,g}(T,z) \colon (T,z) \in {\rm SL}_n(\R) \times \R^n\}
\end{equation}
has an optimizer.
\end{optimization}

Our first result shows that an optimizer of Problem~\ref{q:optimization} indeed exists. 

\begin{thm}\label{t:exitence}
For every fixed pair of functions $f,g \in \mathcal{L}^n$,  there exists a pair $(T_0,z_0) \in {\rm SL}_n(\R) \times \R^n$ solving the optimization problem \eqref{e:opt1}.
\end{thm}

\begin{rmk} It is worth noting that the above theorem holds for nonnegative compactly supported, continuous functions and for nonnegative, continuous, integrable functions that are bounded with sufficient decay at infinity.
\end{rmk}


The main definition of this paper is the following maximal intersection position for log-concave functions.

\begin{deff}\label{mainDef}
We say that two functions $f,g \in \mathcal{L}^n$ are in \emph{maximal intersection position} if the pair $(I_n,o) \in {\rm SL}_n(\R) \times \R^n$ is a solution to the optimization problem \eqref{e:opt1}; that is, 
\[
P_{f,g}(T,z) \leq P_{f,g}(I_n,o) 
\]
for all $(T,z) \in {\rm SL}_n(\R) \times \R^n$. 
\end{deff}

\begin{rmk}
    Let $K$ and $L$ be convex bodies in $\R^n$. Choosing characteristic functions $f=\mathbbm{1}_K$ and $g=\mathbbm{1}_L$ in Definition \ref{mainDef}, it follows that for any $(T,z)\in M_n(\R)\times\R^n$, we have $P_{\mathbbm{1}_K,\mathbbm{1}_L}(T,b)=\vol_n(K\cap (TL+z))$. Thus, Problem \eqref{e:opt1} and Definition \ref{mainDef} reduce to the maximal intersection position of convex bodies studied in \cite{AK2018,AP-2022,BH-2022}.
\end{rmk}

\subsection{John-type theorem for the maximal intersection position}

Before stating our main result, we require one more technical definition: 

\begin{deff} We say that a pair $f,g \in \mathcal{L}^n$ satisfy the \emph{support regularity property} if the following hold: 
\begin{itemize}
    \item $\vol_{n-1}(\partial K_f \cap \partial K_g) = 0$;
    \item $\vol_{n-1}(\partial K_f \cap \partial (TK_g+z)) = 0$ for all but finitely many $(T,z) \in  {\rm SL}_n(\R) \times \R^n$;
    \item $\nabla g$ has bounded second moment, that is, $\int_{K_g}|\nabla g(x)|^2\,dx<\infty$.
\end{itemize}
\end{deff}

Our main result is the following isotropicity theorem for the maximal intersection position of log-concave functions. 

\begin{thm}\label{t:isotropic}
Suppose that $f,g = e^{-\varphi}\in\mathcal{L}^n$ satisfy the support regularity property. If $f$ and $g$ are in maximal intersection position, then the following statements hold and are equivalent: 
\begin{itemize}
    \item[(i)] For any matrix $T \in M_n(\R)$, we have
\begin{equation}\label{e:1stisotopriccondition}
\begin{split}
&\int_{K_f \cap \intt(K_g)}f(x)g(x)\langle \nabla \varphi(x),Tx \rangle\, dx + \int_{K_f \cap \partial K_g}f(x)g(x)\langle n_{K_g}(x),Tx \rangle\, d\sigma_{\partial K_g}(x)\\
 &=\frac{\tr(T)}{n}\left[\int_{K_f \cap \intt(K_g)} f(x)g(x) \langle \nabla \varphi(x), x\rangle\, dx + \int_{K_f \cap \partial K_g}f(x)g(x)\langle n_{K_g}(x),x \rangle\, d\sigma_{\partial K_g}(x) \right].
\end{split}
\end{equation}

\item[(ii)] For every $\theta \in \mathbb{S}^{n-1}$, we have  
\begin{equation}\label{e:3rdisotropiccondition}
 \begin{split}
&\int_{K_f \cap \intt(K_g)}f(x)g(x)\langle \nabla \varphi(x),\theta \rangle \langle x, \theta \rangle\, dx + \int_{K_f \cap \partial K_g}f(x)g(x)\langle n_{K_g}(x),\theta\rangle \langle x, \theta \rangle\, d\sigma_{\partial K_g}(x)\\
 &=\frac{1}{n}\left[\int_{K_f \cap \intt(K_g)} f(x)g(x) \langle \nabla \varphi(x), x\rangle\, dx + \int_{K_f \cap \partial K_g}f(x)g(x)\langle n_{K_g}(x),x \rangle\, d\sigma_{\partial K_g}(x) \right],
\end{split}   
\end{equation}
and for every $\theta, \omega \in \mathbb{S}^{n-1}$ with $\langle \theta,\omega \rangle = 0$, 
\[
\int_{K_f \cap \intt(K_g)}f(x)g(x) \langle \nabla \varphi(x),\theta \rangle \langle x,\omega \rangle\, dx = - \int_{K_f \cap \partial K_g}f(x)g(x) \langle n_{K_g}(x), \theta \rangle \langle x,\omega \rangle \,d\sigma_{\partial K_g}(x).
\]
\end{itemize}

Moreover, 
\begin{equation}\label{e:centering}
o= \int_{K_f \cap \intt(K_g)}f(x)g(x) \nabla \varphi(x)\, dx+ \int_{K_f \cap \partial K_g}f(x)g(x)n_{K_g}(x)\, d\sigma_{\partial K_g}(x).
\end{equation}
\end{thm}

Observe that there is no requirement that the log-concave functions in Theorem \ref{t:isotropic} have compact supports. This is a key feature of Theorem \ref{t:isotropic} that distinguishes it from related results in the literature, in particular from those in the geometric setting of \cite{AK2018,AP-2022}, and those in the functional setting of \cite{IN-2023},   where it was assumed that at least one of the functions has bounded support.

The proof of Theorem \ref{t:isotropic} will be divided into two steps. First, we prove the result for log-concave functions with compact supports by deriving a formula for the first variation of the maximal intersection position. The arguments in this step are inspired by those in the geometric settings of \cite{AK2018,AP-2022}. In the second step, we remove the compact support restriction via an approximation argument. The proof of Theorem~\ref{t:isotropic} also relies on two variational formulas, stated in Theorems~\ref{t:variationalformula} and \ref{t:shifts} below.  

\begin{rmk}
    The integral expressions appearing in Theorem \ref{t:isotropic} are reminiscent of those in Rotem's  formula \cite[Theorem 1.5]{Rotem-2023} for the first variation $\delta(f,g)=\lim_{t\to 0^+}\frac{\int f\oplus(t\times g)-\int f}{t}$ of a log-concave function $g\in\mathcal{L}^n$ in the ``direction" $f\in\mathcal{L}^n$ (see also \cite{Colesanti-Fragala-2013}), where, for $x\in\R^n$ and $t>0$,  $[f\oplus(t\times g)](x)=\sup_{x = y +  tz} f(y)g(z)^t$ is the Asplund sum (sup-convolution) of $f$ and $g$. There it was shown that \[\delta(f,g)=\int_{\R^n}h_g\,d\mu_f+\int_{\Sp}h_{K_g}\,d\nu_f,\] where $\mu_f$ and $\nu_f$ are certain measures on $\R^n$ and $\Sp$, respectively, that depend on $f$, and $h_g$ is the support function of $g$. This formula arose naturally as a consequence of the co-area formula applied to functions of bounded variation and integration by parts. Integration by parts is also used in our proof of Theorem \ref{t:isotropic}, leading to the singular and nonsingular terms appearing there.  For more background and other closely related formulas for first variations of log-concave functions, we refer the reader to, e.g., \cite{Artstein-Rubinstein,Colesanti-Fragala-2013,Fang-Zhang,HLXZ-JDG-2023,Rotem1,Rotem-2023,Ulivelli-2023} and the references therein. 
\end{rmk}

\begin{rmk}
    Many log-concave functions satisfy the second moment condition in Theorem \ref{t:isotropic}, including gaussians and characteristic functions of convex bodies. It is unclear to us if this assumption is  necessary, or if it is simply an  artifact of our proof method.  We leave this question as a topic for future research.
\end{rmk}

 The following corollary is immediate from \eqref{e:3rdisotropiccondition} and \eqref{e:centering}. 

\begin{corollary} 
Suppose that $f, g = e^{-\varphi}\in \mathcal{L}^n$ satisfy  the hypotheses of Theorem~\ref{t:isotropic}. If $f$ and $g$ are in maximal intersection position, then the following decomposition of the identity holds:  
\[
C_{f,g} I_n = \int_{K_f \cap \intt(K_g)}\nabla \varphi \otimes x f(x)g(x) \, dx + \int_{K_f \cap \partial K_g} n_{K_g}(x) \otimes x f(x)g(x)\, d\sigma_{\partial K_g}(x) ,
\]
where
\[
C_{f,g}:= \frac{1}{n} \left(\int_{K_f \cap \intt(K_g)}  \langle \nabla \varphi(x), x\rangle f(x)g(x) \,dx + \int_{ K_f \cap \partial K_g} \langle n_{K_f}(x),x \rangle f(x)g(x)\, d\sigma_{\partial K_f}(x)\right),
\]
provided $C_{f,g} \neq 0$. 
Moreover, 
\[
o= \int_{K_f \cap \intt(K_g)} \nabla \varphi(x) f(x)g(x)\, dx+ \int_{K_f \cap \partial K_g}n_{K_g}(x) 
 f(x)g(x)\,d\sigma_{\partial K_g}(x).   
\]
\end{corollary}

\begin{proof} 
Observe that the condition \eqref{e:3rdisotropiccondition} of Theorem~\ref{t:isotropic} can be rephrased in the following way: for every $\theta \in \mathbb{S}^{n-1}$, 
\[
\int_{K_f \cap \intt(K_g)} \theta^t\nabla \varphi(x) x^t\theta f(x)g(x)\,dx + \int_{\partial K_f \cap K_g} \theta^tn_K(x) x^t\theta \,d\sigma_{\partial K}(x) = C_{f,g}. 
\]
In other words, for every $\theta \in \mathbb{S}^{n-1}$, we must have 
\[
\theta^t \left(\frac{1}{C_{f,g}} \left[\int_{K_f \cap \intt(K_g)}f(x)g(x) \nabla \varphi \otimes x \,dx + \int_{\partial K_f \cap K_g}f(x)g(x) n_{K_f}(x) \otimes x \,d\sigma_{\partial K_f}(x)  \right]\right) \theta = 1.
\]
Consequently, 
\[
I_n = \frac{1}{C_{f,g}} \left[\int_{K_f \cap \intt(K_g)}f(x)g(x) \nabla \varphi(x) \otimes x  \,dx + \int_{\partial K_f \cap K_g}f(x)g(x) n_{K_f}(x) \otimes x \,d\sigma_{\partial K_f}(x)  \right],
\]
as required. 
\end{proof}

\begin{rmk}
    Let $K$ and $L$ be convex bodies in $\R^n$ which satisfy $\vol_{n-1}(\partial K\cap\partial L)=0$ and $\vol_{n-1}(\partial K\cap\partial (TL+b))$ for all but finitely many $(T,b)\in{\rm SL}_n(\R)\times\R^n$. For such $K$ and $L$, choose $f=\mathbbm{1}_K$ and $g=\mathbbm{1}_L$. Then $f$ and $g$  satisfy the hypotheses of Theorem \ref{t:isotropic}, so under these restrictions  on the bodies $K$ and $L$, we recover  \cite[Theorem 1.4]{AP-2022} (stated as Theorem \ref{AP-thm} above); see also  \cite[Theorem 1.3]{AK2018} and \cite[Theorem 6.6]{AP-2022}.
\end{rmk}
\section{Proof of Theorem \ref{t:exitence}}\label{s:existence}

Next, we demonstrate that the problem \eqref{e:opt1} is sensible by showing that it always admits a solution.

\begin{proof}[Proof of Theorem~\ref{t:exitence}] Let $f,g \in \mathcal{L}^n$, and fix $(T,b) \in {\rm SL}_n(\R) \times \R^n$. Under the present assumptions on $f$ and $g$, there exist constants $a,c>0$ and $b,d \in \R$ such that 
\begin{equation}\label{e:linearconvex}
-\log(f(x)) \geq c|x| +d \qquad\text{and}\qquad -\log(g(x)) \geq a|x|+b
\end{equation}
for all $x \in \R^n$ (see, for example, \cite[Theorem 3.2.6]{Rockafellar-Wets}). Therefore, after a change of variables, we get
\begin{equation}\label{e:one}
\begin{split}
P_{f,g}(T,b) &\leq e^{-(b+d)} \int_{\R^n} e^{-(a|x| + c|Tx+z|)}\,dx \\
&\leq C(a,b,c,d) \int_{\R^n} e^{-(|x|+|Tx+z|)}\,dx,
\end{split}
\end{equation}
where $C(a,b,c,d) >0$ is some constant. 

Consider the singular value decomposition $T=UDV$ for some $n$-dimensional orthogonal matrices $U$ and $V$ and diagonal matrix $D = {\rm diag}(\lambda_1,\dots,\lambda_n)$, where $\lambda_i \geq \lambda_{i+1} > 0$, $i=1,\dots,n-1$, are the eigenvalues of $T$. Since the integral is invariant under orthogonal transformations, using the change of variables $y=Vx$ we obtain 
\begin{equation}\label{e:two}
\begin{split}
&\int_{\R^n}e^{-(|x| + |Tx+z|)}\, dx = \int_{\R^n} e^{-(|x| + |UDVx + z|)}\,dx = \int_{\R^n}e^{-(|y| + |Dy + U^{-1} z|)}\, dy\\
&\leq \int_{\R^n} \exp \left\{- \frac{1}{\sqrt{n}} \sum_{i=1}^n (|y_i| + |\lambda_iy_i + z_i'|) \right\} dy = C(n) \prod_{i=1}^n \int_{\R} e^{-(|y_i| + |\lambda_i y_i + z_i'|)}\,dy_i,
\end{split}
\end{equation}
where $z_i' = (U^{-1}z)_i$ and $C(n) >0$ is some constant. Combining \eqref{e:one} and \eqref{e:two}, we have shown that
\[
P_{f,g}(T,b) \leq C(a,b,c,d,n) \prod_{i=1}^n \int_{\R} e^{-(|y_i| + |\lambda_i y_i + z_i'|)}\,dy_i.
\]

For $\lambda >0$ and $s \in \R$, consider the integral 
\[
I_{\lambda,s} := \int_{\R} e^{-(|t|+|\lambda t + s| )} \,dt = \int_{\R} e^{-(|t| + |\lambda t + |s||)}\,dt. 
\]
To continue our estimate, we will compute $I_{\lambda,s}$ explicitly. We begin by writing 
\begin{align*}
I_{\lambda,s} &= \int_{\R} e^{-(|t| + |\lambda t + |s||)}\,dx\\
&= \underbrace{\int_{-\infty}^{-|s|/\lambda} e^{(1+\lambda)t + |s|}\,dt}_{=:I_1} + \underbrace{\int_{-|s|/\lambda}^{0} e^{(1-\lambda)t - |s|}\,dt}_{=:I_2} + \underbrace{\int_{0}^{\infty} e^{-(1+\lambda)t - |s|}\,dt}_{=:I_3}.
\end{align*}

\noindent It can be readily checked that 
\[
I_1 = \frac{e^{-|s|/\lambda}}{1+\lambda} \quad \text{ and } \quad I_3 = \frac{e^{-|s|}}{1+\lambda}.
\]
For $I_2$, we have 
\[
I_2 = e^{-|s|}\int_{-|s|/\lambda}^{0} e^{(1-\lambda)t} \,dt=\begin{cases}
    |s|e^{-|s|}, &\text{if }\lambda=1;\\
    \frac{e^{-|s|}-e^{-|s|/\lambda}}{1-\lambda}, &\text{if }\lambda\neq 1.
\end{cases}
\]
 
\noindent Hence, 
\[
I_{\lambda,s} = \begin{cases}
    (|s|+1)e^{-|s|}, &\text{if }\lambda=1;\\
    \frac{2\left(e^{-|s|}-\lambda e^{-|s|/\lambda}\right)}{1-\lambda^2}, &\text{if }\lambda\neq 1.
\end{cases}
\]

 Putting everything together, we have shown that for an arbitrary pair $(T,z) \in {\rm SL}_n(\R) \times \R^n$, after writing the singular value decomposition $T = UDV$  of $T$  as above, we obtain the estimate 
\begin{align*}
P_{f,g}(T,z) 
&\leq C(a,b,c,d,n)\prod_{i=1}^n I_{\lambda_i,b_i'}\\
&=C(a,b,c,d,n)\prod_{\lambda_i=1}I_{1,z_i'}\prod_{\lambda_i\neq 1}I_{\lambda_i,z_i'}.
\end{align*}
The function $h(s)=(|s|+1)e^{-|s|}$ has global maximum $h(0)=1$, and for every fixed $\lambda>0$, the function $h_\lambda(s)=\frac{2\left(e^{-|s|}-\lambda e^{-|s|/\lambda}\right)}{1-\lambda^2}$ has global maximum $h_\lambda(0)=\frac{2}{1+\lambda}$.  Therefore, $\prod_{\lambda_i=1}I_{1,b_i'}\leq 1$ and hence
\begin{equation}\label{upper-est1}
P_{f,g}(T,z)\leq C(a,b,c,d,n)\prod_{\lambda_i\neq 1}I_{\lambda_i,z_i'}
\leq C(a,b,c,d,n)\prod_{\lambda_i\neq 1}\frac{2}{1+\lambda_i}.
\end{equation}

Consider any sequence $\{(T_k,z_k)\}_k \subset {\rm SL}_n(\R) \times \R^n$, where  $\lambda_i(k)$ denotes the $i$th eigenvalue of $T_k$ (after its singular value decomposition), with $\lambda_i(k) \geq \lambda_{i+1}(k) > 0$, for all $i=1,\dots,n-1$.  Assume that this sequence satisfies 
\[
P_{f,g}(T_k,z_k)  \to \sup\{P_{f,g}(T,z) \colon (T,z) \in {\rm SL}_n(\R) \times \R^n \} =: m_{f,g}>0
\]
as $k \to \infty$. Then the maximal entry of $D_k={\rm diag}(\lambda_1(k),\ldots,\lambda_n(k))$ tends to infinity since the orthogonal group $O_n$ and the set of diagonal matrices with entries bounded by $c$ are compact sets. 
By \eqref{upper-est1}, we then have that 
\[
\lim_{k \to \infty}P_{f,g}(T_k,z_k) \leq C(a,b,c,d,n) \lim_{k\to \infty} \prod_{\lambda_i(k)\neq 1} \frac{2}{1+\lambda_i(k)} = 0,
\]
a contradiction.  Thus, there exists a constant $M >0$ such that $|\lambda_i(k)| \leq M$ for all $i,k$, and hence the matrices $T_k$ all belong to some compact set.  Using this fact, together with a similar argument to the one given above, from \eqref{e:one} it follows that the $|z_k|$ are uniformly bounded for all $k$. Consequently, there exists a subsequence $\{(T_{k_l},z_{k_l})\}_{l}$ of the original sequence converging to some $(T_0,z_0) \in {\rm SL}_n(\R) \times \R^n$. Thus, by the continuity of the map $T\mapsto \int_{\R^n}f(x)g(T^{-1}x)\,dx$ on ${\rm SL}_n(\R^n)$, we obtain 
\[
\lim_{l \to \infty} P_{f,g}(T_{k_l},z_{k_l}) = P_{f,g}(T_0,z_0) = m_{f,g},
\]
as desired.  This shows that the optimization problem \eqref{e:opt1} indeed admits a solution. 
\end{proof}

\begin{rmk} We note that it is only necessary to have that $g \in \mathcal{L}^n$ in the above proof, since if $f$ is any $\log$-concave function, then $\|f\|_{\infty} < \infty$, and the above computation follows through in this case. The roles of $f$ and $g$ may be reversed. 
\end{rmk}
\section{Variational formulas and the proof of Theorem~\ref{t:isotropic}} \label{s:variationalformulas}

The key ingredients in the proof of Theorem~\ref{t:isotropic} are the following variational formulas. In the case of compactly supported functions, the proofs of these results will combine ideas of Artstein-Avidan and Katzin \cite{AK2018} and Artstein-Avidan and Putterman \cite{AP-2022} (see also \cite{Fang-Zhang}). 

\begin{thm}\label{t:variationalformula}
Let $f,g\in\mathcal{L}^n$ satisfy the support regularity property. For $T \in M_n(\R)$ and $\e>0$, set  $T_\e := \frac{I_n + \e T}{\det(I_n +\e T)^{1/n}}$ and consider the function $F \colon [0,\infty) \to \R$ defined by 
\[
F(\e) =F_{f,g,T}(\e):= \int_{\R^n}f(x)g(T_\e^{-1}x)\, dx.
\] 
Then 
\begin{align*}
\frac{dF(\e)}{d\e} &= \int_{\intt(T_\e^{-1}K_f) \cap K_g} f(T_\e x) \langle \nabla g(x) ,A_\e x\rangle\, dx\\
&-\int_{(T_\e^{-1}K_f) \cap \partial K_g} f(T_\e x) g(x) \langle n_{K_g}(x), A_\e x\rangle\, d\sigma_{\partial K_g}(x)
    \end{align*}
whenever $0 \leq \e \leq \e_0$, where $A_\e := \left(\frac{\tr(T)}{n}I_n-T\right)T_\e +O(\e)T_\e$. In particular,
\begin{align*}
\frac{dF(\varepsilon)}{d\varepsilon}\bigg|_{\varepsilon=0}&=\int_{K_f \cap \intt(K_g)} f(x) \left\langle \nabla g(x), \left(T-\frac{\tr(T)}{n}I_n\right)x \right\rangle dx\\
&-\int_{K_f \cap \partial K_g}g(x)f(x)\left\langle n_{K_g}(x),\left(\frac{\tr(T)}{n}I_n-T\right)x\right\rangle d\sigma_{\partial K_g}(x).
    \end{align*}
\end{thm}

\begin{thm} \label{t:shifts} Let $f,g  \in \mathcal{L}^n$. For each $y \in \R^n$, define $E_y \colon [0,\infty) \to \R$ by 
\[
E_y(\e)=E_{f,g,y}(\e) := \int_{\R^n}f(x)g(y-\e y)\, dy. 
\]
Assume that $\vol_{n-1}(\partial K_f \cap \partial K_g ) = 0$, 
$\vol_{n-1}(\partial (K_f + \e y) \cap K_g) = 0$  for every $0 \leq \e \leq \e'$, and $\int_{K_g}|\nabla g(x)|^2\,dx<\infty$. Then 
\begin{align*}
\frac{d E_y(\e)}{d\e} &= \int_{\intt(K_f- \e y) \cap K_g} f(x+\e y) \langle \nabla g(x), y \rangle\, dx\\
&- \int_{(K_f -\e y) \cap \partial K_g} f(x+\e y) g(x) \langle n_{K_g}(x),y \rangle \,d\sigma_{\partial K_g}(x).
\end{align*}
In particular, we have 
\begin{align*}
\frac{d E_y(\e)}{d\e}\bigg|_{\varepsilon=0} &= \int_{K_f \cap \intt(K_g)} f(x) \langle \nabla g(x), y \rangle\, dx
- \int_{K_f \cap \partial K_g} f(x) g(x) \langle n_{K_g}(x),y \rangle \,d\sigma_{\partial K_g}(x).
\end{align*}
\end{thm}

Since the proofs are rather involved, we postpone them until the following subsection.  However, assuming the validity of Theorem~\ref{t:variationalformula} and Theorem~\ref{t:shifts}, we are in a position to prove Theorem~\ref{t:isotropic}. 

\begin{proof}[Proof of Theorem~\ref{t:isotropic}] 
We begin with the proof of identity \eqref{e:1stisotopriccondition}. Fix $T \in M_n(\R)$, and recall the function $F(\e) = \int_{K_f \cap (T_\e K_g)}f(x) g(T_\e^{-1}x) \,dx$. Since $f$ and $g$ are in maximal intersection position, we have that $\frac{dF(\varepsilon)}{d\varepsilon}\big|_{\varepsilon=0} =0$, whereby Theorem~\ref{t:variationalformula} immediately implies \eqref{e:1stisotopriccondition}.  

To see that \eqref{e:1stisotopriccondition} implies \eqref{e:3rdisotropiccondition}, for each $i,j \in \{1,\dots,n\}$, consider $T_{i,j} = e_i \otimes e_j$.  Then \eqref{e:1stisotopriccondition} implies that 
\begin{align*}
&\int_{K_f \cap \intt(K_g)}f(x)g(x)\langle \nabla \varphi(x),e_i\rangle \langle x, e_j \rangle\, dx + \int_{K_f \cap \partial K_g}f(x)g(x)\langle n_{K_g}(x),e_i\rangle \langle x, e_j\rangle\, d\sigma_{\partial K_g}(x)\\
 &=\frac{\delta_{i,j}}{n}\left[\int_{K_f \cap \intt(K_g)} f(x)g(x) \langle \nabla \varphi(x), x\rangle\, dx + \int_{K_f\cap \partial K_g}f(x)g(x)\langle n_{K_g}(x),x \rangle\, d\sigma_{\partial K_g}(x) \right],
\end{align*}
where $e_k = (0,\dots,0,1,0,\dots,0)$ denotes the $k$th elementary basis vector in $\R^n$, and $\delta_{i,j}$ is the Kronecker delta defined by $\delta_{i,j} = 1$ if $i=j$ and $\delta_{i,j} = 0$ otherwise. Applying a rotation if necessary, \eqref{e:3rdisotropiccondition} follows. 

To see that \eqref{e:3rdisotropiccondition} implies \eqref{e:1stisotopriccondition}, it is enough to observe that if $T \in M_n(\R)$ is given by $T = (t_{ij})_{i,j=1}^n$, then 
\[
\langle \nabla \varphi(x), Tx \rangle = \sum_{i,j=1}^n t_{ij}(\nabla \varphi(x))_ix_j \quad \text{ and } \quad \langle x, Tx \rangle = \sum_{i,j=1}^n t_{ij}x_ix_j. 
\]

Finally, again because $f$ and $g$ are in maximal intersection position,  Theorem~\ref{t:shifts} implies that the identity
\begin{align*}
0 &=\frac{d}{d\varepsilon}\bigg|_{\varepsilon=0}\left[\int_{K_f \cap (K_g +\e y)} f(x) g(x-\e y)\, dx\right]\\
&= \int_{K_f \cap \intt(K_g)} g(x) f(x) \langle \nabla \varphi(x),y \rangle\, dx  + \int_{K_f \cap \partial K_g}f(x)g(x) \langle n_{K_g}(x),y\rangle \,d\sigma_{\partial K_g}(x)
\end{align*}
holds for every $y \in \R^n$, which yields equation \eqref{e:centering}. 
\end{proof}


\section{Proof of Theorems~\ref{t:variationalformula} and \ref{t:shifts}}
In this section, we provide proofs of Theorems~\ref{t:variationalformula} and \ref{t:shifts}, which we break into two cases: compactly supported $\log$-concave functions, and $\log$-concave functions with unbounded supports. As we shall see, the proof of the case of compactly supported log-concave functions holds more generally for all bounded, sufficiently smooth functions whose supports are sufficiently regular. 

\subsection{Compactly supported functions}

We begin with the  special case of compactly supported functions. 

\begin{thm} \label{t:vari}
Let $f,g \in \mathcal{L}^n$ be such that $K_f$ and $K_g$ are compact. Given $T \in M_{n}(\R)$ and $\e>0$, set $T_{\e}:=\frac{I_n+\e T}{\det(I_n+\e T)^{1/n}}$ and consider the function $F \colon [0,\infty) \to \R$ defined by 
\[
F(\e) := \int_{K_f \cap (T_{\e}K_g)} f(x) g(T_{\e}^{-1}x )\, dx. 
\] 
If $f$ and $g$ satisfy the support regularity property, then
\begin{align*}
\frac{dF(\varepsilon)}{d\varepsilon}&=\int_{(T_{\e}^{-1} K_f) \cap \intt(K_g)}f(T_\e x)\langle \nabla g(x), A_\e x \rangle \,dx\\ 
&- \int_{(T_{\e}^{-1} K_f) \cap \partial K_g} g(x) f(T_\e x) \langle n_{K_g}(x), A_\e x \rangle\, d\sigma_{\partial K_g} (x)
    \end{align*}
whenever $0 < \e \leq \e_0$, where $A_\e := \left(\frac{\tr(T)}{n}I_n-T\right)T_\e +O(\e)T_\e$. 
\end{thm}

\begin{proof}[Proof of Theorem~\ref{t:vari}] 

Consider sequences $\{\varphi_j\}_{j\in\mathbb{N}}$ and $\{\psi_j\}_{j\in\mathbb{N}}$  of continuously differentiable functions  approximating $g\mathbbm{1}_{K_g}$ and  $f\mathbbm{1}_{K_f}$, respectively,  defined by 
\[
\varphi_j(x) := \begin{cases}
g(x), &\text{if } \|x\|_{K_g} \leq 1- \frac{1}{j},\\
g_j(x), &\text{if } 1-\frac{1}{j} \leq \|x\|_{K_g} \leq 1,\\
0, &\text{if } \|x\|_{K_g} \geq 1,
\end{cases}
\,\text{  and  }\,
\psi_j(x) := \begin{cases}
f(x), &\text{if } \|x\|_{K_f} \leq 1- \frac{1}{\sqrt{j}},\\
f_j(x), &\text{if } 1-\frac{1}{\sqrt{j}} \leq \|x\|_{K_f} \leq 1,\\
0, &\text{if } \|x\|_{K_f} \geq 1,
\end{cases}
\]
and which satisfy $|\nabla \varphi_j(x)| <cj$ and $| \nabla \psi_j(x)| <d\sqrt{j}$ for all $x\in\R^n$, where $c,d$ are constants.   
By construction, $\supp(\varphi_j)\subset K_g$ and $\supp(\psi_j)\subset K_f$ for every $j\in\mathbb{N}$. Moreover, as $j \to \infty$, the sequences $\{\varphi_j\}$ and $\{\psi_j\}$ converge pointwise to $g\mathbbm{1}_{K_g}$ and  $f\mathbbm{1}_{K_f}$, respectively. Hence, for every $\e >0$ we have 
\[
\frac{dF(\e)}{d\e}  = \frac{d}{d\e} \int_{\R^n} \lim_{j \to \infty}\varphi_j(T_\e^{-1}x) \psi_j(x)\, dx.
\]
To prove the theorem, it suffices to verify that  the following identities hold for all sufficiently small $\e>0$: 
\begin{equation}\label{e:dct}
\frac{d}{d\e} \int_{\R^n} \lim_{j \to \infty} \varphi_j(T_\e^{-1}x) \psi_j(x)\, dx = \frac{d}{d\e} \lim_{j \to \infty}\int_{\R^n} \varphi_j(T_\e^{-1}x) \psi_j(x)\, dx;
\end{equation}
\begin{equation}\label{e:commutation}
\frac{d}{d\e} \lim_{j \to \infty}\int_{\R^n}  \varphi_j(T_\e^{-1}x) \psi_j(x)\, dx = \lim_{j \to \infty} \frac{d}{d\e}\int_{\R^n} \varphi_j(T_\e^{-1}x) \psi_j(x)\, dx;
\end{equation}
\begin{equation}\label{e:derivative}
\frac{d}{d\e} \int_{\R^n} \varphi_j(T_\e^{-1}x) \psi_j(x) \,dx = \int_{\R^n} \left\langle \nabla \varphi_j(x),\left(\frac{\tr(T)}{n}I_n-T\right)T_\e x +O(\e)x\right\rangle \psi_j(T_\e x) \,dx;
\end{equation}
\begin{equation}\label{e:unifconvg}
\begin{split}
&\lim_{j \to \infty} \int_{\R^n}\left\langle \nabla \varphi_j(x), \left(\frac{\tr(T)}{n}I_n-T\right)T_\e^{-1}x +O(\e)x \right\rangle \psi_j(T_\e x)\, dx\\
&= \int_{(T_{\e}^{-1} K_f) \cap \intt(K_g)}f(T_\e x) \left\langle \nabla g(x), \left(T- \frac{\tr(T)}{n}I_n \right)x  + O(\e)x\right\rangle dx\\
&- \int_{(T_\e^{-1}K_f)\cap\partial K_g} g(x) f(T_\e x) \left\langle n_{K_g}(x),\left(\frac{\tr(T)}{n}I_n-T\right)T_\e^{-1}x +O(\e)x \right\rangle d\sigma_{\partial K_g} (x). 
\end{split}
\end{equation}

Identity \eqref{e:dct} follows from the Lebesgue dominated convergence theorem. Next, recall that for any matrix $A \in M_n(\R)$ with $\|A\| < 1$, one has that $(I_n - A)^{-1} = I_n + A + \sum_{m =2}^\infty A^m$, where $\|A\|$ is the operator norm of $A$. Applying this identity with $A = -\e T$, with $\e>0$ chosen sufficiently small we obtain  
\begin{align*}
T_{\e}^{-1} &= \det(I_n+\e T)^{\frac{1}{n}}(I_n+\e T)^{-1}= \left(1 + \e\frac{\tr(T)}{n} + O(\e^2) \right)(I_n-\e T +O(\e^2))\\
&=I_n +\e \left(\frac{\tr(T)}{n}I_n - T \right) + O(\e^2). 
\end{align*}
Therefore, for every $x \in \R^n$ and all $\e>0$, 
\[
\frac{d}{d\e} (T_\e^{-1}x) = \frac{\tr(T)}{n}I_nx - Tx + O(\e)x. 
\]
This observation, together with the chain rule,  implies that for each fixed $j \in \N$,
\[
\frac{d}{d\e}(\varphi_j(T_\e^{-1}x)) = \left\langle \nabla \varphi_j(T_\e^{-1}x),\frac{\tr(T)}{n}I_nx - Tx + O(\e)x\right\rangle.
\]
Thus \eqref{e:derivative} follows from the change of variables $y = T_\e^{-1}x$. Noting that $\det(T_\e^{-1}) = 1$, to establish \eqref{e:commutation} and \eqref{e:unifconvg}, we need only to verify  the following proposition.

\begin{proposition} \label{p:prop1} For $T\in M_n(\R)$ and $\e>0$, define $A_\e := \left(\frac{\tr(T)}{n}I_n-T\right)T_\e +O(\e)T_\e$.  Under the hypotheses of Theorem \ref{t:vari}, there exists $\e_0 > 0$ such that 
\[
 \int_{\R^n} \langle \nabla \varphi_j(x), A_\e x \rangle \psi_j(T_\e x)\, dx
\]
converges uniformly to 
\[
\int_{T_\e^{-1}K_f\cap \intt(K_g)}\langle \nabla g(x),A_\e x\rangle f(T_\e x)\,dx-\int_{T_\e^{-1}K_f\cap\partial K_g}f(T_\e x)g(x)\langle n_{K_g}(x),A_\e x\rangle\,d\sigma_{\partial K_g}(x)\
\]
as $j \to \infty$ whenever $0 < \e \leq \e_0$.
\end{proposition}

\begin{proof}[Proof of Proposition~\ref{p:prop1}] 
For each $j\in \N$, define the sets 
\[
M_j(g) := \left\{x \in \R^n \colon 1 - \frac{1}{j} \leq \|x\|_{K_g} \leq 1 \right\}
\]
and
\[
\quad M_j(f) := \left\{x \in \R^n \colon 1 - \frac{1}{\sqrt{j}} \leq \|x\|_{K_f} \leq 1 \right\}.
\]
Then for each $j\in\mathbb{N}$, the following inclusions hold:
\begin{align*}
M_j(g) &\supset \supp(\varphi_j -g\mathbbm{1}_{K_g})\supset \supp(\nabla(\varphi_j -  g\mathbbm{1}_{K_g}))\\
M_j(f) &\supset\supp(\psi_j - f\mathbbm{1}_{K_f})\supset \supp(\nabla(\psi_j - f\mathbbm{1}_{K_f})).
\end{align*}
Integrating by parts, we obtain
\begin{align}\label{algebra}
    &\int_{\R^n}\langle\nabla\varphi_j(x), A_\e x\rangle\psi_j(T_\e x)\,dx-\int_{T_\e^{-1}K_f\cap \intt(K_g)}\langle\nabla g(x),A_\e x\rangle f(T_\e x)\,dx \nonumber\\
    &=\int_{\R^n}\langle\nabla\varphi_j(x), A_\e x\rangle\psi_j(T_\e x)\,dx-\int_{\R^n}\langle \nabla g(x),A_\e x\rangle f(T_\e x)\,dx \nonumber \\
    &=\int_{M_j(g)}\langle\nabla(\varphi_j-g)(x),A_\e x\rangle\psi_j(T_\e x)\,dx+\int_{M_j(g)}\langle\nabla g(x),A_\e x\rangle\left[\psi_j(T_\e x)-f(T_\e x)\right]dx \nonumber\\
   &=I_1(j,\e)-I_2(j,\e)+I_3(j,\e),
\end{align}
where
\begin{align*}
    I_1(j,\e)&:=\int_{\partial M_j(g)}\psi_j(T_\e x)\left[\varphi_j(x)-g(x)\right]\langle n_{M_j(g)}(x),A_\e x\rangle\,d\sigma_{\partial M_j(g)}(x)\\
    I_2(j,\e)&:=\int_{\partial M_j(g)}\psi_j(T_\e x)\left[\varphi_j(x)-g(x)\right]\langle n_{M_j(g)}(x),A_\e x\rangle\,d\sigma_{\partial M_j(g)}(x) \\
    I_3(j,\e) &:=\int_{T_\e^{-1}M_j(f)\cap M_j(g)}\langle\nabla g(x),A_\e x\rangle\left[\psi_j(T_\e x)-f(T_\e x)\right]dx.
\end{align*}

Starting with $I_1(j,\e)$, observe that
 \[
\partial M_j(g) = \partial K_g\sqcup \left(\frac{j-1}{j}\right)\partial K_g,
 \]
 where $\varphi_j(x)-g(x)=-g(x)$ for $x\in \partial K_g$ and $\varphi_j(x)-g(x)=0$ for $x\in\left(\frac{j-1}{j}\right)\partial K_g$. Hence,
 \begin{align*}
     I_1(j,\e) &= -\int_{T_\e^{-1}K_f\cap\partial K_g}\psi_j(T_\e x)g(x)\langle n_{K_g}(x),A_\e x\rangle\,d\sigma_{\partial K_g}(x)\\
     &=-\int_{T_\e^{-1}K_f\cap\partial K_g}f(T_\e x)g(x)\langle n_{K_g}(x),A_\e x\rangle\,d\sigma_{\partial K_g}(x)\\
     &+\int_{T_\e^{-1}K_f\cap\partial K_g}\left[f(T_\e x)-\psi_j(T_\e x)\right]g(x)\langle n_{K_g}(x),A_\e x\rangle\,d\sigma_{\partial K_g}(x).
 \end{align*}
 
\noindent Therefore, by \eqref{algebra} we have
 \begin{align}\label{mainstep-prop}
     &\int_{\R^n}\langle\nabla\varphi_j(x), A_\e x\rangle\psi_j(T_\e x)\,dx-\nonumber\\
     &\left(\int_{T_\e^{-1}K_f\cap K_g}\langle g(x),A_\e x\rangle f(T_\e x)\,dx-\int_{T_\e^{-1}K_f\cap\partial K_g}f(T_\e x)g(x)\langle n_{K_g}(x),A_\e x\rangle\,d\sigma_{\partial K_g}(x)\right)\nonumber\\
&=\underbrace{\int_{T_\e^{-1}K_f\cap\partial K_g}\left[f(T_\e x)-\psi_j(T_\e x)\right]g(x)\langle n_{K_g}(x),A_\e x\rangle\,d\sigma_{\partial K_g}(x)}_{=:J_1(j,\e)}+I_2(j,\e)+I_3(j,\e).
 \end{align}

Continuing with $J_1(j,\e)$, by the Cauchy-Schwarz inequality, 
\begin{align}\label{J1-est}
    |J_1(j,\e)| &=\left|\int_{T_\e^{-1}M_j(f)\cap\partial K_g}\left[f(T_\e x)-\psi_j(T_\e x)\right]g(x)\langle n_{K_g}(x),A_\e x\rangle\,d\sigma_{\partial K_g}(x)\right|\nonumber\\
    &\leq 2\|f\|_\infty\|g\|_\infty\max_{x\in K_g}|A_\e x|\xi_F(j,\e)
\end{align}
where $\xi_F(j,\varepsilon):=\vol_{n-1}(T_\e^{-1}M_j(f)\cap\partial K_g)$ is monotonically decreasing in $j$ for every fixed $\e>0$. By Dini's theorem, for any fixed $\varepsilon>0$ which is sufficiently small,  $\xi_F(j,\varepsilon)$ converges uniformly to $\vol_{n-1}(\partial T_\e^{-1} K_f\cap\partial K_g)$ as $j\to\infty$. Therefore, assuming that $\vol_{n-1}(\partial T_\e^{-1} K_f\cap \partial K_g) = 0$ for all but finitely many $\e>0$, we may find some $\e_2 >0$ such that $\xi_F(j,\e) = 0$ whenever $\e \leq \e_2$, which implies that $I_1(j,\e)$ converges uniformly to zero provided $0 < \e \leq \e_2$. We then choose $\e_0 = \min\{\e_1,\e_2\}$.

Next, we estimate $I_2(j,\e)$. By the product rule for divergence, 
\begin{align*}
\divv\big(\psi_j(T_\e x)A_\e x\big)
&= \langle \nabla \psi_j(T_\e x), A_\e x\rangle + \psi_j(T_\e x)\divv(A_\e x)\\
&=\langle \nabla \psi_j(T_\e x), A_\e x\rangle +\psi_j(T_\e x)\tr(A_\e)
\end{align*}
where $\tr(A_\e)=O(\e)$. Therefore, by the  Cauchy-Schwarz inequality we get
\begin{align}\label{I2-est-prop}
    |I_2(j,\e)| &=\left|\int_{\intt(M_j(g))}\left[\varphi_j(x)-g(x)\right]\left(\langle \nabla \psi_j(T_\e x), A_\e x\rangle +\psi_j(T_\e x)\tr(A_\e)\right)dx\right|\nonumber\\
    &\leq 2\|g\|_\infty\left(d\sqrt{j}\max_{x\in K_g}|A_\e x|+2\|f\|_\infty O(\e)\right)\vol_n(M_j(g))\nonumber\\
    &=2\|g\|_\infty\left(d\sqrt{j}\max_{x\in K_g}|A_\e x|+2\|f\|_\infty O(\e)\right)\left[1-\left(1-\frac{1}{j}\right)^n\right]\vol_n(K_g)\nonumber\\
    &\leq 2\|g\|_\infty\left(d\sqrt{j}\max_{x\in K_g}|A_\e x|+2\|f\|_\infty O(\e)\right)\cdot\frac{n}{j}\vol_n(K_g).
\end{align}

It remains to estimate $I_3(j,\e)$. We have 
\begin{align}\label{I3-est-prop}
    |I_3(j,\e)| &\leq \int_{T_\e^{-1}M_j(f)\cap M_j(g)}\big|\langle\nabla g(x),A_\e x\rangle\left(\psi_j(T_\e x)-f(T_\e x)\right)\big|\,dx \nonumber\\
    &\leq\left(\int_{T_\e^{-1}M_j(f)\cap M_j(g)}\langle\nabla g(x),A_\e x\rangle^2\,dx\right)^{\frac{1}{2}}
    \left(\int_{T_\e^{-1} M_j(f)\cap M_j(g)}\left(\psi_j(T_\e x)-f(T_\e x)\right)^2\right)^{\frac{1}{2}} \nonumber\\
    &\leq 2\|f\|_\infty \left(\max_{x\in K_g}|A_\e x|^2\right)^{\frac{1}{2}}\left(\int_{K_g}|\nabla g(x)|^2\,dx\right)^{\frac{1}{2}}\vol_n(M_j(g))\nonumber\\
    &\leq 2\|f\|_\infty \left(\max_{x\in K_g}|A_\e x|^2\right)^{\frac{1}{2}}\left(\int_{K_g}|\nabla g(x)|^2\,dx\right)^{\frac{1}{2}}\frac{n}{j}\vol_n(K_g).
\end{align}

Finally, applying the triangle inequality in \eqref{mainstep-prop} and then substituting the estimates \eqref{J1-est}, \eqref{I2-est-prop} and \eqref{I3-est-prop}, we finally obtain that for all $\e\in(0,\e_0]$,
\begin{align*}
&\bigg|\int_{\R^n}\langle\nabla \varphi_j(x), A_\e x\rangle\psi_j(T_\e x)\,dx-\\
     &\left(\int_{T_\e^{-1}K_f\cap \intt(K_g)}\langle \nabla g(x),A_\e x\rangle f(T_\e x)\,dx-\int_{T_\e^{-1}K_f\cap\partial K_g}f(T_\e x)g(x)\langle n_{K_g}(x),A_\e x\rangle\,d\sigma_{\partial K_g}(x)\right)\bigg|\\
     &\leq |J_1(j,\e)|+|I_2(j,\e)|+|I_3(j,\e)|\\
     &\leq 2\|f\|_\infty\|g\|_\infty\max_{x\in K_g}|A_\e x|\xi_F(j,\e)+2\|g\|_\infty\left(d\sqrt{j}\max_{x\in K_g}|A_\e x|+2\|f\|_\infty O(\e)\right)\cdot\frac{n}{j}\vol_n(K_g)\\
     &+2\|f\|_\infty \left(\max_{x\in K_g}|A_\e x|^2\right)^{\frac{1}{2}}\left(\int_{K_g}|\nabla g(x)|^2\,dx\right)^{\frac{1}{2}}\frac{n}{j}\vol_n(K_g)\\
     &\longrightarrow 0\quad \text{as}\quad j\to\infty.
\end{align*}
In the last step, we used the assumption that $\int_{K_g}|\nabla g(x)|^2\,dx<\infty$.
    \end{proof}
\noindent This concludes the proof of Theorem \ref{t:vari}.
\end{proof}

Both Theorem~\ref{t:exitence} and Theorem~\ref{t:vari} hold in a more general setting, with the same proofs.  We record this version of Theorem~\ref{t:vari} here (compare with \cite[Theorem 6.6]{AP-2022}). 

\begin{thm} Let $f,g \colon \R^n \to [0,\infty)$ be compactly supported functions (each support having piecewise smooth boundary) that are of differentiability class $C^1$ on the interior of their supports. Given $T \in M_{n}(\R)$ and $\e>0$, set $T_{\e}:=\frac{I_n+\e T}{\det(I_n+\e T)^{1/n}}$ and consider the function $F \colon [0,\infty) \to \R$ defined by 
\[
F(\e) := \int_{K_f \cap (T_{\e}K_g)} f(x) g(T_{\e}^{-1}x )\, dx. 
\] 
Assume  that $\vol_{n-1}(\partial K_f\cap\partial K_g)=0$,
$\vol_{n-1}(\partial (T_{\e}^{-1} K_f) \cap\partial K_g)=0$ for all but finitely many $\e >0$, and $\int_{K_g}|\nabla g(x)|^2\,dx<\infty$. Then
\begin{align*}
\frac{dF(\varepsilon)}{d\varepsilon}&=\int_{(T_{\e}^{-1} K_f) \cap \intt(K_g)}f(T_\e x) \langle \nabla g(x), A_\e x \rangle \,dx\\ 
&- \int_{(T_{\e}^{-1} K_f) \cap \partial K_g} g(x) f(T_\e x) \langle n_{K_g}(x), A_\e x \rangle\, d\sigma_{\partial K_g} (x),
    \end{align*}
whenever $0 < \e \leq \e_0$, where $A_\e := \left(\frac{\tr(T)}{n}I_n-T\right)T_\e +O(\e)T_\e$. 
\end{thm}


\subsection{The general case}

Now we move to the general setting of Theorem~\ref{t:variationalformula}, where the functions considered may have unbounded supports. We require the following lemma. 

\begin{lemma}\label{l:integral} For every $\varphi\in\mathcal{L}^n$, every convex set $M \subset \R^n$, and all $A,T \in {\rm SL}_n(\R)$, we have
\begin{align*}
&\Gamma(n) \int_{\partial (TM)} \varphi(x) \langle n_{TM}(x), A x \rangle\, d\sigma_{\partial (TM)}(x) \\
&= \int_{\R^n} \left|\nabla e^{-\|x\|_{TM}}\right| \varphi\left(\frac{x}{\|x\|_{TM}}\right)\left\langle n_{ TM}\left(\frac{x}{\|x\|_{TM}}\right), A \frac{x}{\|x\|_{TM}} \right\rangle dx.
\end{align*}
    
\end{lemma}

\begin{proof}
By the definition of the gamma function, we may write
\begin{align*}
\Gamma(n) \int_{\partial (TM)} \varphi(x) \langle n_{TM}(y), A y \rangle\, d\sigma_{\partial (TM)}(y) &= \int_0^{\infty} \int_{\partial (TM)}  \varphi(y) \langle n_{TM}(y), A y \rangle  e^{-s}\mathbb{S}^{n-1}\, dy\, ds\\
&= \int_0^{\infty} \int_{\partial (sTM)} e^{-s} \varphi\left(\frac{x}{s}\right) \left\langle n_{sTM}\left(\frac{x}{s}\right), A\frac{x}{s} \right\rangle dx\, ds,
\end{align*}
where we made the change of variable $x=sy$. Now since $x \in \partial (sTM)$, we have $s = \|x\|_{TM}$. Consequently, for each $s >0$,
\[
e^{-s} = e^{-\|x\|_{TM}} = \frac{\left| \nabla e^{-\|x\|_{TM}} \right|}{| \nabla \|x\|_{TM}|}.
\]
Applying the co-area formula, we obtain
\begin{align*}
\Gamma(n)&\int_{\partial (TM)} \varphi(x) \langle n_{TM}(x), A x \rangle\, d\sigma_{\partial (TM)}(x) \\
&=\int_0^{\infty} \int_{\partial (sTM)} e^{-s} \varphi\left(\frac{x}{s}\right) \left\langle n_{sTM}\left(\frac{x}{s}\right), A\frac{x}{s} \right\rangle dx\, ds\\
&=\int_0^{\infty} \int_{\partial (sTM)} \varphi\left(\frac{x}{\|x\|_{TM}}\right) \left\langle n_{TM}\left(\frac{x}{\|x\|_{TM}}\right), A\left(\frac{x}{\|x\|_{TM}}\right) \right\rangle  \frac{\left| \nabla e^{-\|x\|_{TM}} \right|}{| \nabla \|x\|_{TM}}\, dx\, ds\\
&=\int_0^{\infty} \int_{\partial (TM)} \varphi\left(\frac{sy}{\|sy\|_{TM}}\right) \left\langle n_{TM}\left(\frac{sy}{\|sy\|_{TM}}\right), A\left(\frac{sy}{\|sy\|_{TM}}\right) \right\rangle  \frac{\left| \nabla e^{-\|sy\|_{TM}} \right|}{| \nabla \|sy\|_{TM}|}\mathbb{S}^{n-1} \,dy\, ds\\
&=\int_{\R^n} \varphi\left(\frac{z}{\|z\|_{TM}}\right) \left\langle n_{TM}\left(\frac{z}{\|z\|_{TM}}\right), A\left(\frac{z}{\|z\|_{TM}}\right) \right\rangle  \left| \nabla e^{-\|z\|_{TM}} \right| dz.
\end{align*}
In the last line, we made use of the polar decomposition $z = sy$ where $s > 0$ and $y \in \partial K_s(f)$, so that 
\[
dz = \frac{\mathbb{S}^{n-1}}{|\nabla \|sy\|_{TM}|}\,dy\,ds,
\]
as required. 
\end{proof}


\begin{proof}[Proof of Theorem~\ref{t:variationalformula}] 

Since $f$ and $g$ are log-concave, so too is $fg$. Hence, there exists $x_0 \in \R^n$ such that $\|fg\|_{\infty}=f(o)g(o)$.  For each $\delta > 0$, consider the following truncations of $f$ and $g$ (see \cite{AKSW}):
\[
f_{\delta}(x) := f(x) e^{-\delta |x|^2} \mathbbm{1}_{\{f > \delta \}}(x) \quad \text{ and } \quad g_{\delta}(x) := g(x) e^{-\delta |x|^2} \mathbbm{1}_{\{g > \delta \}}(x).
\]
For each $\delta >0$, we set $K_f^\delta := \supp(f_\delta)$ and $K_g^\delta := \supp(g_{\delta})$. Note that the following conditions may be verified: 
\begin{itemize}
    \item[(i)] For each fixed $x \in \R^n$, $f_{\delta}(x) \uparrow f(x)$ and $g_{\delta}(x) \uparrow g(x)$ as $\delta \to 0$. Moreover, $K_f^\delta \uparrow K_f$ and $K_g^\delta \uparrow K_g$ as $\delta \to 0$.
    
    \item[(ii)] For every $\delta >0$, $f_{\delta}$ and  $g_\delta $ are log-concave functions. Thus,  $f_\delta\in C^1(\intt(K_f^\delta))$, $g_\delta \in C^1(\intt(K_g^\delta))$ and $x_0 \in \intt(K_f^\delta \cap K_g^\delta)$, so we may assume that $o \in \intt(K_f^\delta \cap K_g^\delta)$ for all $\delta >0$. 
    
    \item[(iii)] There exists some sufficiently small $\delta_0 >0$  such that for every  $\delta \leq \delta_0$, we have 
    \[
    \vol_{n-1}(\partial K_f^\delta \cap \partial K_g^{\delta}) = 0, \vol_{n-1}(\partial K_f^\delta \cap K_g^\delta) > 0, \text{ and } \vol_{n-1}(\partial K_f^\delta \cap \partial (T_\e K_g^\delta)) =0
    \]
    for all but finitely many $\varepsilon> 0$.

    \item[(iv)] For every $\delta>0$, we have $\int_{K_g^\delta}|\nabla g_\delta(x)|^2\,dx<\infty$.
\end{itemize}




We will only prove (iv). Observe that  
\begin{equation}\label{gradient-g-delta}
\nabla g_\delta(x)=(\nabla g(x)e^{-\delta|x|^2}-2\delta g(x)e^{-\delta|x|^2}x)\mathbbm{1}_{K_g^\delta}(x).
\end{equation}
Hence, 
\begin{align*}
    &\int_{K_g^\delta}|\nabla g_\delta(x)|^2\,dx\\ =&\int_{K_g^\delta}\bigg(|\nabla g(x)|^2 e^{-2\delta|x|^2}-4\delta g(x) e^{-2\delta|x|^2}\left\langle \nabla g(x),  x\right\rangle +4\delta^2 g(x)^2 e^{-2\delta|x|^2}|x|^2\bigg)\,dx\\
    &\leq \int_{K_g^\delta}|\nabla g(x)|^2 e^{-2\delta|x|^2}\,dx+4\delta\|g\|_\infty \int_{K_g^\delta}e^{-2\delta|x|^2}|\langle\nabla g(x),x\rangle|\,dx+4\delta^2\|g\|_\infty^2\int_{K_g^\delta}|x|^2 e^{-2\delta|x|^2}\,dx\\
    &\leq \int_{K_g}|\nabla g(x)|^2\,dx+4\delta\|g\|_\infty\left(\int_{K_g}|\nabla g(x)|^2\,dx\right)^{\frac{1}{2}}\left(\int_{K_g}|x|^2 e^{-2\delta|x|^2}\,dx\right)^{\frac{1}{2}}\\
    &+4\delta^2\|g\|_\infty^2\int_{K_g}|x|^2 e^{-2\delta|x|^2}\,dx.
\end{align*}
By assumption, $\int_{K_g}|\nabla g(x)|^2\,dx<\infty$, and we always have  $\int_{K_g}|x|^2 e^{-2\delta|x|^2}\,dx<\infty$. Thus, (iv) is proved.

\vspace{2mm}

Consider the function $F_\delta(\e) := \int_{K_f^\delta \cap (T_\e K_g^\delta)} f_\delta(x)g_\delta(T_\e^{-1}x) \,dx$. As in the proof of Theorem~\ref{t:vari},  for  $\e>0$ we define $A_\e := \left(\frac{\tr(T)}{n}I_n-T\right)T_\e +O(\e)T_\e$. In view of  properties (ii)-(iv), for each $\delta >0$  and $0 < \e \leq \e_0$ sufficiently small, we may apply Theorem~\ref{t:vari} to $f_\delta$ and $g_\delta$ to obtain 
\begin{align*}
\frac{dF_\delta(\varepsilon)}{d\varepsilon}&=\int_{(T_\e^{-1} K_f^\delta) \cap \intt(K_g^\delta)}f_\delta(T_\e x) \langle \nabla g_\delta(x), A_\e x \rangle \,dx\\
&- \int_{(T_\e^{-1} K_f^\delta) \cap \partial K_g^\delta} f_\delta(T_\e x)g_\delta(x) \langle n_{K_g^\delta}(x), A_\e x \rangle\, d\sigma_{\partial K_g^\delta} (x).
    \end{align*}
Thus, for $\delta >0$ and $0 < \e \leq \e_0$, 
\begin{align*}
&\left|\frac{dF(\e)}{d\e}-\frac{dF_\delta(\e)}{d\e}\right|=\\
&=\left|\int_{\intt(T_\e^{-1}K_f) \cap K_g}f(T_\e\cdot) \left\langle \nabla g, A_\e \right\rangle dx -\int_{(T_\e^{-1} K_f^\delta) \cap \intt(K_g^\delta)}f_\delta(T_\e \cdot) \left\langle \nabla g_\delta, A_{\e} \right\rangle dx \right.\\
&-\left.\int_{ (T_\e^{-1}K_f) \cap \partial K_g}f(T_\e\cdot)g \langle n_{K_g},A_\e \rangle d\sigma_{\partial K_g}\right. \left.+  \int_{(T_\e^{-1} K_f^\delta) \cap \partial K_g^\delta}f_\delta(T_\e \cdot) g_\delta \langle n_{K_g^\delta},A_\e\rangle d\sigma_{\partial K_g^{\delta}} \right|\\
&\leq E_1(\delta,\e) + E_2(\delta, \e),\\
\end{align*}
where 
\[
E_1(\delta,\e):=\left|\int_{\intt(T_\e^{-1}K_f) \cap K_g}f(T_\e\cdot) \left\langle \nabla g, A_\e \right\rangle dx -\int_{(T_\e^{-1} K_f^\delta) \cap \intt(K_g^\delta)}f_\delta(T_\e \cdot) \left\langle \nabla g_\delta, A_{\e} \right\rangle dx  \right|
\]
and 
\begin{align*}
E_2(\delta, \e) &:= \left|\int_{ (T_\e^{-1}K_f) \cap \partial K_g}f(T_\e\cdot)g \langle n_{K_g},A_\e \rangle d\sigma_{\partial K_g}\right. \left.-  \int_{(T_\e^{-1} K_f^\delta) \cap \partial K_g^\delta}f_\delta(T_\e \cdot) g_\delta \langle n_{K_g^\delta},A_\e\rangle d\sigma_{\partial K_g^{\delta}} \right|.
\end{align*}

Therefore, using \eqref{gradient-g-delta} we get 
\begin{align}\label{unbounded-case-1}
    E_1(\delta,\e) &= \int_{\intt(T_\e^{-1}K_f) \cap K_g}f(T_\e x) \left\langle \nabla g(x), A_\e x \right\rangle dx -\int_{(T_\e^{-1}K_f^\delta)\cap \intt( K_g^\delta)}f(T_\e x)\langle\nabla g(x),A_\e x\rangle\,dx \nonumber\\
&+\int_{(T_\e^{-1}K_f^\delta)\cap \intt( K_g^\delta)}f(T_\e x)\langle\nabla g(x),A_\e x\rangle\,dx
-\int_{(T_\e^{-1} K_f^\delta) \cap \intt(K_g^\delta)}f_\delta(T_\e x) \left\langle \nabla g_\delta(x), A_{\e}x \right\rangle dx \nonumber\\
&=\int_{\Delta_\delta(\e)}f(T_\e x)\langle\nabla g(x),A_\e x\rangle\,dx \nonumber\\
&+\int_{(T_\e^{-1} K_f^\delta) \cap \intt(K_g^\delta)}\left(f(T_\e x)-f_\delta(T_\e x)e^{-\delta|x|^2}\right) \left\langle \nabla g(x), A_{\e}x \right\rangle dx \nonumber\\
&+2\delta\int_{(T_\e^{-1}K_f^\delta)\cap \intt( K_g^\delta)}f_\delta(T_\e x)g(x)e^{-\delta|x|^2}\langle x,A_\e x\rangle \,dx
\end{align}
where $\Delta_\delta(\e):= [(T_\e^{-1}K_f) \cap \intt(K_g)] \setminus [(T_\e^{-1} K_f^\delta) \cap \intt(K_g^\delta)]$.  By the Cauchy-Schwarz inequality, we obtain the estimate 
\begin{align*}
    \left|\int_{\Delta_\delta(\e)}f(T_\e x)\langle\nabla g(x),A_\e x\rangle\,dx\right|
    &\leq \int_{\Delta_\delta(\e)}|f(T_\e x)|\cdot|\langle\nabla g(x),A_\e x\rangle|\,dx\\
    &\leq\left(\int_{\Delta_\delta(\e)}f(T_\e x)^2\,dx\right)^{\frac{1}{2}}\left(\int_{\Delta_\delta(\e)}\langle\nabla g(x),A_\e x\rangle^2\,dx\right)^{\frac{1}{2}}\\
    &\leq \|f\|_\infty \left(\max_{x\in \Delta_\delta(\e)}|A_\e x|^2\right)^{\frac{1}{2}}\left(\int_{\Delta_\delta(\e)}|\nabla g(x)|^2\right)^{\frac{1}{2}}\vol_n(\Delta_\delta(\e))^{\frac{1}{2}}.
\end{align*}
By hypothesis, $\int_{K_g}|\nabla g|^2\,dx$ is bounded.  Because of the second part of property (i), for every fixed $\e>0$ we have that $\Delta_\delta(\e) \to \varnothing$ as $\delta \to 0$, which implies 
\[
\left|\int_{\Delta_\delta(\e)}f(T_\e x)\langle\nabla g(x),A_\e x\rangle\,dx\right| \to 0 \text{ as } \delta \to 0.
\]
Furthermore, by the monotone convergence theorem,
\[
\int_{(T_\e^{-1} K_f^\delta) \cap \intt(K_g^\delta)}\left(f(T_\e x)-f_\delta(T_\e x)e^{-\delta|x|^2}\right) \left\langle \nabla g(x), A_{\e}x \right\rangle dx \to 0\quad \text{as}\quad \delta\to 0.
\]

Thus, going back to \eqref{unbounded-case-1} we have
\begin{align*}
    |E_1(\delta,\e)| &\leq \left|\int_{\Delta_\delta(\e)}f(T_\e x)\langle\nabla g(x),A_\e x\rangle\,dx\right| \\
&+\left|\int_{(T_\e^{-1} K_f^\delta) \cap \intt(K_g^\delta)}\left(f(T_\e x)-f_\delta(T_\e x)e^{-\delta|x|^2}\right) \left\langle \nabla g(x), A_{\e}x \right\rangle dx\right| \\
&+2\delta\int_{(T_\e^{-1}K_f^\delta)\cap \intt( K_g^\delta)}f_\delta(T_\e x)g(x)e^{-\delta|x|^2}|\langle x,A_\e x\rangle| \,dx\\
&\leq\left|\int_{\Delta_\delta(\e)}f(T_\e x)\langle\nabla g(x),A_\e x\rangle\,dx\right| \\
&+\left|\int_{(T_\e^{-1} K_f^\delta) \cap \intt(K_g^\delta)}\left(f(T_\e x)-f_\delta(T_\e x)e^{-\delta|x|^2}\right) \left\langle \nabla g(x), A_{\e}x \right\rangle dx\right| \\
&+2\delta\|f\|_\infty\|g\|_\infty\max_{x\in K_g}|A_\e x|\int_{(T_\e^{-1}K_f^\delta)\cap \intt( K_g^\delta)}|x|e^{-\delta|x|^2}\,dx\\
&\longrightarrow 0 \quad \text{as}\quad \delta\to 0.
\end{align*}
Therefore, $E_1(\delta,\e) \to 0$ as $\delta \to 0$ for every $\e \in [0, \e_0]$. 

Next, we consider $E_2(\delta,\e)$. For each $\e>0$, we set $\varphi_\e (x) := f(T_\e x)g(x)\mathbbm{1}_{K_g}(x)$, and for each $\delta >0$ we  set $\varphi_{\delta,\e}(x) := f_\delta(T_\e x) g_\delta(x)\mathbbm{1}_{K_g^\delta}(x)$. Applying  Lemma~\ref{l:integral} to each of the integrals comprising $E_2(\delta,\e)$, and then integrating in polar coordinates over $S=S(f,g):= \mathbb{S}^{n-1} + x_0$, we may write 
\begin{equation}
\begin{split}
E_2(\delta,\e) &= \frac{1}{\Gamma(n)}\left|\int_{\R^n} \varphi_\e\left(\frac{x}{\|x\|_{K_g}}\right) \left\langle n_{K_g}\left(\frac{x}{\|x\|_{K_g}}\right), A_\e\frac{x}{\|x\|_{K_g}} \right\rangle \left|\nabla e^{-\|x\|_{K_g}}\right| dx \right.\\
&\left. - \int_{\R^n} \varphi_{\delta,\e }\left(\frac{x}{\|x\|_{K_g^\delta}}\right) \left\langle n_{K_g^{\delta}}\left(\frac{x}{\|x\|_{K_g^\delta}}\right), A_\e\frac{x}{\|x\|_{K_g^\delta}} \right\rangle \left|\nabla e^{-\|x\|_{K_g^\delta}}\right| dx \right|\\
&=\frac{\Gamma(n+1)}{\Gamma(n)}\left|\int_{S} \varphi_\e\left(\frac{u}{\|u\|_{K_g}}\right) \left\langle n_{K_g}\left(\frac{u}{\|u\|_{K_g}}\right), A_\e\frac{u}{\|u\|_{K_g}} \right\rangle \|u\|_{K_g}^{-n} |\nabla \|u\|_{K_g}|\, du \right.\\
&\left. - \int_{S} \varphi_{\delta,\e}\left(\frac{u}{\|u\|_{K_g^\delta}}\right) \left\langle n_{K_g^\delta}\left(\frac{u}{\|u\|_{K_g^\delta}}\right), A_\e\frac{u}{\|u\|_{K_g^\delta}} \right\rangle \|u\|_{K_g^\delta}^{-n} |\nabla \|u\|_{K_g^\delta}| \,du \right|.
\end{split}
\end{equation}
This can be estimated from above by
\begin{equation}\label{e:estimate1}
\begin{split}
E_2(\delta,\e)&\leq n \max_{u \in S} \left\{\left|\left\langle n_{K_g}\left(\frac{u}{\|u\|_{K_g}}\right), A_\e\frac{u}{\|u\|_{K_g}} \right\rangle \right| |\nabla \|u\|_{K_g}|\right\}\\
&\times \int_{S}   \left|\varphi_\e\left(\frac{u}{\|u\|_{K_g}}\right)-\varphi_{\delta,\e}\left(\frac{u}{\|u\|_{K_g^\delta}}\right)\right|du\\
&+n \|\varphi_\e\|_{\infty} \int_{S} \left|\left\langle n_{K_g}\left(\frac{u}{\|u\|_{K_g}}\right), A\frac{u}{\|u\|_{K_g}} \right\rangle \|u\|_{K_g}^{-n} |\nabla \|u\|_{K_g}|\right.\\
&\left.- \left\langle n_{K_g^\delta}\left(\frac{u}{\|u\|_{K_g^\delta}}\right), A\frac{u}{\|u\|_{K_g^\delta}} \right\rangle \|u\|_{K_g^\delta}^{-n} |\nabla \|u\|_{K_g^\delta}| \right| du.
\end{split}
\end{equation}

Property (i) implies that $K_g^\delta \uparrow K_g$ as $\delta \to 0$, thereby implying that $\frac{u}{\|u\|_{K_g^\delta}} \to \frac{u}{\|u\|_{K_g}}$  as $\delta \to 0$ whenever $u \in S$. This, together with the fact that $\varphi_{\delta, \e}(u) \uparrow  \varphi_\e$, implies that  $\varphi_{\delta,\e} \left(\frac{u}{\|u\|_{K_g^\delta}}\right) \to \varphi_{\e} \left(\frac{u}{\|u\|_{K_g}}\right)$ as $\delta \to 0$ for each $u \in S$.  Thus, the dominated convergence theorem yields that  
\[
\int_{S} \left|\varphi_{\e} \left(\frac{u}{\|u\|_{K_g}}\right)-\varphi_{\delta,\e} \left(\frac{u}{\|u\|_{K_g^\delta}}\right) \right| du \to 0 \quad\text{as}\quad \delta \to 0,
\]
so the penultimate term of \eqref{e:estimate1} tends to zero as $\delta \to 0$.  Therefore, it remains to estimate the last integral appearing in equation \eqref{e:estimate1}. We have
\begin{equation}\label{e:estimate2}
\begin{split}
&\int_{S} \left|\left\langle n_{K_g}\left(\frac{u}{\|u\|_{K_g}}\right), A_\e\frac{u}{\|u\|_{K_g}} \right\rangle \|u\|_{K_g}^{-n} |\nabla \|u\|_{K_g}|\right.
\left.\right.\\
&\left.- \left\langle n_{K_g^\delta}\left(\frac{u}{\|u\|_{K_g^\delta}}\right), A_\e\frac{u}{\|u\|_{K_g^\delta}} \right\rangle \|u\|_{K_g^\delta}^{-n} |\nabla \|u\|_{K_g^\delta}| \right| du\\
&\leq\underbrace{\int_{S} \left| \left\langle n_{K_g}\left(\frac{u}{\|u\|_{K_g}}\right), A_\e\frac{u}{\|u\|_{K_g}} \right\rangle \right| \cdot\left|\|u\|_{K_g}^{-n} \cdot |\nabla\|u\|_{K_g}| - \|u\|_{K_g^\delta}^{-n} \cdot |\nabla \|u\|_{K_g^\delta}|\right|du}_{=:L_1(\delta,\e)} \\
&+ \underbrace{\int_{S} \|u\|_{K_g^\delta}^{-n} \big|\nabla \|u\|_{K_g^\delta}\big| \left|\left\langle n_{K_g}\left(\frac{u}{\|u\|_{K_g}}\right) - n_{K_g^\delta}\left(\frac{u}{\|u\|_{K_g^\delta}}\right), A_\e\left(\frac{u}{\|u\|_{K_g}} - \frac{u}{\|u\|_{K_g^\delta}}\right) \right\rangle  \right| du}_{=:L_2(\delta,\e)}.\\
\end{split}
\end{equation}

We estimate $L_1(\delta,\e)$ and $L_2(\delta,\e)$ individually.  By \cite[Theorem~12.35]{Rockafellar-Wets},  for every $u \in S$ we have $\nabla \|u\|_{K_g^\delta } \to \nabla \|u\|_{K_g}$ as $\delta \to 0$. Using this fact and property (i), by the Cauchy-Schwarz inequality we obtain  
\begin{align*}
L_1(\delta,\e) &\leq \sqrt{ \int_S\left|\|u\|_{K_g}^{-n} \cdot |\nabla\|u\|_{K_g}| - \|u\|_{K_g^\delta}^{-n} \cdot |\nabla \|u\|_{K_g^\delta}|\right|^2 du}\\
&\times \sqrt{\int_{S} \left| \left\langle n_{K_g}\left(\frac{u}{\|u\|_{K_g}}\right), A_\e\frac{u}{\|u\|_{K_g}} \right\rangle \right|^2 du} \to 0
\end{align*}
as $\delta \to 0$. Furthermore, applying the Cauchy-Schwarz inequality again, we derive that 
\begin{align*}
L_2(\delta,\e) &\leq \max_{u \in S}\left|A_\e\left(\frac{u}{\|u\|_{K_g}} - \frac{u}{\|u\|_{K_g^\delta}}\right) \right| \sqrt{\int_{S}\|u\|_{K_g^\delta}^{-n} \cdot |\nabla \|u\|_{K_g^\delta}|^2\,du} \\
&\times \sqrt{ \int_{S} \left| n_{K_g}\left(\frac{u}{\|u\|_{K_g}}\right) - n_{K_g^\delta}\left(\frac{u}{\|u\|_{K_g^\delta}}\right)  \right|^2 du} \to 0
\end{align*}
as $\delta\to 0$.

In view of the inequalities \eqref{e:estimate1} and \eqref{e:estimate2}, we see that $E_2(\delta,\e) \to 0$ as $\delta \to 0$ whenever $\e \in [0,\e_0]$. Consequently, for every $\e \in [0,\e_0]$,
\begin{align*}
\frac{d F_\delta(\e)}{d \e} \to &\int_{\intt(T_\e^{-1}K_f) \cap K_g} f(T_\e x) \langle \nabla g(x) ,A_\e x\rangle\, dx\\
&-\int_{(T_\e^{-1}K_f) \cap  \partial K_g} f(T_\e x) g(x) \langle n_{K_g}(x), A_\e x\rangle\, d\sigma_{\partial K_g}(x)
\end{align*}
as $\delta \to 0$.
Since this convergence is uniform,  by \cite[Theorem~14.2.7]{Taobook} we deduce that 
\begin{align*}
\frac{dF(\e)}{d\e} &= \int_{(T_{\e}^{-1} K_f) \cap \intt(K_g)} f(T_\e x) \langle \nabla g(x) ,A_\e x\rangle\, dx\\
&-\int_{(T_\e^{-1}K_f) \cap \partial K_g} f(T_\e x) g(x) \langle n_{K_g}(x), A_\e x\rangle\, d\sigma_{\partial K_g}(x)
\end{align*}
whenever $\e \in [0,\e_0]$, as desired. 
\end{proof}


\subsection{Proof of Theorem~\ref{t:shifts}}

In this subsection, we sketch the proof of Theorem~\ref{t:shifts}. 

\begin{proof} 
The proof of Theorem \ref{t:shifts} is nearly  identical to that of Theorem~\ref{t:vari}. Here we only sketch the proof of the case when the supports $K_f$ and $K_g$ are compact sets; the general case then follows from an approximation argument similar to the one in the proof of Theorem~\ref{t:variationalformula}.  Recall that $f$ and $g = e^{-\varphi}$ belong to $\mathcal{L}^n$, and  we assume that the following conditions hold:
\begin{itemize}
    \item $\vol_{n-1}(\partial K_f \cap \partial K_g) = 0$;
    \item $\vol_{n-1}(\partial K_f \cap \partial (TK_g+z)) = 0$ for all but finitely many $(T,z) \in  {\rm SL}_n(\R) \times \R^n$;
    \item $\nabla g$ has bounded second moment, that is, $\int_{K_g}|\nabla g(x)|^2\,dx<\infty$.
\end{itemize}

We again consider the functions $\varphi_j \uparrow g\mathbbm{1}_{K_g}$ and $\psi_j \uparrow f1_{K_f}$ as in the proof of Theorem~\ref{t:vari}. As before, we need to show that the following identities hold  for all sufficiently small $\e>0$: 

\begin{itemize}
    \item[(i)]  $\frac{d}{d\e} \int_{\R^n} \lim_{j \to \infty}\varphi_j(x-\e y)\psi_j(x)\, dx = \frac{d}{d\e}\lim_{j \to \infty} \int_{\R^n} \varphi_j(x-\e y) \psi_j(x) \,dx$;
    
    \item[(ii)] $\frac{d}{d\e}\lim_{j \to \infty} \int_{\R^n} \varphi_j(x-\e y) \psi_j(x)\, dx = \lim_{j \to \infty} \frac{d}{d\e} \int_{\R^n} \varphi_j(x-\e y) \psi_j(x)\, dx$;
    
    \item[(iii)] $\frac{d}{d\e} \int_{\R^n} \varphi_j(x-\e y) \psi_j(x)\,dx = \int_{\R^n} \psi_j(x+\e y) \langle \nabla \varphi_j(x), -y \rangle\, dx$;
    
    \item[(iv)]
    \begin{align*}
    \lim_{j \to \infty}\int_{\R^n} \psi_j(x+\e y) \langle \nabla \varphi_j(x), y \rangle\, dx &= \int_{(K_f-\e y)\cap \intt(K_g)}f(x+\e y) \langle \nabla g(x), y \rangle\, dx\\
    &- \int_{(K_f-\e y) \cap \partial K_g} f(x+\e y)g(x) \langle n_{K_g}(x), y \rangle \,d\sigma_{\partial K_g}(x). 
    \end{align*}
\end{itemize}

\noindent Equation (i) follows from the dominated convergence theorem, and equation  (iii) follows from the chain rule and a change of variables.  To prove  (ii) and (iv), we need to verify the following 

\begin{proposition} \label{c:conv} There exists $\e_1 >0$ such that $\int_{\R^n} \psi_j(x+\e y) \langle \nabla \varphi_j(x), y \rangle\, dx$ converges uniformly to 
\[
\int_{(K_f-\e y)\cap \intt(K_g)}f(x+\e y) \langle \nabla g(x), y \rangle\, dx- \int_{(K_f-\e y) \cap \partial K_g} f(x+\e y)g(x) \langle n_{K_g}(x), y \rangle \,d\sigma_{\partial K_g}(x).
\]
as $j \to \infty$ whenever $0 < \e \leq \e_1$. 
\end{proposition}

The  proof of Proposition \ref{c:conv} is handled in  the same way as that of Proposition \ref{p:prop1}, with only minor modifications required. Similarly, the rest of the proof of Theorem \ref{t:shifts} is completely analogous to that of Theorem \ref{t:vari}. For the sake of brevity, we omit the details. 
\end{proof}

 \section{An application to a John-type position for even log-concave measures}

The main result of this section is a John-type theorem for $\log$-concave measures on $\R^n$. The methods involved closely follow the arguments in \cite{AK2018}. 

Consider the following variant of the optimization problem \eqref{e:measureopt}. Let $\mu$ be an even $\log$-concave probability measure on $\R^n$ whose density is supported on the whole of $\R^n$, and let $K\subset\R^n$ be an origin-symmetric convex body. For each $r>0$, consider the following quantity:
\begin{equation}\label{e:2measureopt}
m_{\mu}(r) :=\sup\left\{\mu(K \cap TB_2^n): T\in {\rm GL}_n(\R),\,\det(T) = r \right\}.
\end{equation}

Arguing as in the proofs of Theorems \ref{t:exitence} and \ref{t:isotropic}, we derive the following version of Theorem~\ref{t:measure}.

\begin{thm}\label{t:mipmeasures} Let $\mu$ be an even $\log$-concave probability measure on $\R^n$, $K \subset \R^n$ be an origin-symmetric convex body, and $r>0$. 

\begin{enumerate}
    \item[(a)] There exists  $T_0\in{\rm GL}_n(\R)$ such that the supremum \eqref{e:measureopt} is attained, i.e., $m_{\mu}(r) = \mu(K \cap T_0 B_2^n)$. 
\item[(b)] Assume, additionally, that the following hold: 
\begin{itemize}
    \item $\vol_{n-1}(\partial K\cap [r\mathbb{S}^{n-1}]) = 0$;
    \item $\vol_{n-1}(\partial K \cap \partial [TB_2^n]) = 0$ for all but finitely many $(T,z) \in  {\rm GL}_n(\R) \times \R^n$.
\end{itemize} 
If the identity $I_n$ is a solution to \eqref{e:2measureopt} for $K$ and $B_2^n$, then for every $\theta \in \mathbb{S}^{n-1}$,
\begin{equation*}
 \begin{split}
 \int_{K \cap [r\mathbb{S}^{n-1}]}\langle y,\theta\rangle^2 \, d\mu_{[r \mathbb{S}^{n-1}]}(y)
 =\frac{\mu_{r\mathbb{S}^{n-1}}(K \cap [r\mathbb{S}^{n-1}]) }{n}.
\end{split}  
\end{equation*}
In particular, the following decomposition of the identity holds:  
\[
\frac{I_n}{n}=\frac{\int_{K \cap [r\mathbb{S}^{n-1}]} y\otimes y \, d\mu_{r \mathbb{S}^{n-1}}(y) }{\mu_{r\mathbb{S}^{n-1}}(K \cap [r\mathbb{S}^{n-1}])}.
\]
\end{enumerate}
    
\end{thm}


\begin{thm}\label{t:johnnew} Let $\mu$ be an even $\log$-concave probability measure on $\R^n$ whose density is supported on the whole of $\R^n$, and let $K \subset \R^n$ be an origin-symmetric convex body in $\mu$-John position. For each $r >1$, let 
\[
T_r L \in \argmax \{\mu(K \cap TB_2^n) \colon T\in{\rm GL}_n(\R),\,\det(T) = r\} \]
be the image of $B_2^n$ of measure $\mu_r = r \mu(T_rB_2^n)$. Consider the probability measures
\[
\nu_r(A) = \frac{\mu(A \cap [\partial K \setminus T_rB_2^n])}{\mu(\partial K \setminus T_rB_2^n)}. 
\]
If $\vol_{n-1}(\partial K \cap \partial T_r B_2^n)=0$ for all $r$ sufficiently close to $1$, then there exists a sequence $r_j \searrow 1$ such that the sequence $\nu_{r_j}$ converges weakly to a measure $\nu$ supported in $\partial K \cap S^{n-1}$ satisfying the condition
\[
\int_{\partial K \cap S^{n-1}} x \otimes x \,d\nu(x) = \frac{I_n}{n}. 
\]
In other words, the measure $\nu$ on $\partial K \cap \mathbb{S}^{n-1}$ is isotropic. 
\end{thm}

We will need the following technical lemma. 

\begin{lemma} \label{l:tech} Let $\mu$ be an even $\log$-concave probability measure on $\R^n$ with density $\varphi \colon \R^n \to (0,\infty)$, and let $K \subset \R^n$ be an origin-symmetric convex body in $\mu$-John position.  Then the following hold:
\begin{itemize}
    \item[(i)] For $0 < r \leq 1$, we have $m_\mu(r) =\mu(L)$. 
    \item [(ii)] The function $m_\mu(r)$ is increasing and continuous on $[1,\infty)$.
    \item[(iii)] For each $r >1$, let $T_r \in {\rm GL}_n(\R)$ be a solution to \eqref{e:2measureopt}.  Then there exists a subsequence $r_j \searrow 1$ of $\{r > 1\}$ such that the functions $(1_K \varphi) \cdot 1_{T{R_j} B_2^n}$ converge to $(1_K \varphi) \cdot 1_{B_2^n}$ uniformly as $j \to \infty$. 
\end{itemize}
\end{lemma}

\begin{proof} Item (i) is obvious. For brevity, set $f = 1_K \varphi$, where $\varphi$ is the density of $\mu$, and $g = 1_{B_2^n}$. Fix $1 \leq r < s < \infty$. Set $m_{f,g}(r) := m_{\mu}(r)$. 
 On the one hand, observe that there exists a pair $(T_r,z_r)\in {\rm GL}(\R^n)\times\R^n$ such that
\begin{align*}
m_{f,g}(r) &= \int_{\R^n}f(x) g(T_r^{-1}(x-z_r)) \,dx
\leq \int_{\R^n}f(x) g\left(\left[\left(\frac{s}{r}\right)^{\frac{1}{n}}T_r\right]^{-1}\left(x\right)\right)dx
\leq m_{f,g}(s).
\end{align*}
On the other hand, we also have
\begin{align*}
    m_{f,g}(r) &= \int_{\R^n}f(x)g(T_r^{-1}(x))\,dx\\
    &\geq \int_{\R^n} f(x) g\left(\left[\left(\frac{r}{s}\right)^{\frac{1}{n}} T_s\right]^{-1}x\right)dx\\
    &\geq \int_{\R^n} f(x) g\left(T_s^{-1}\left[\left(\frac{s}{r}\right)^{\frac{1}{n}} x\right]\right)dx\\
    &\geq \int_{\R^n}f\left(\left(\frac{s}{r}\right)^{\frac{1}{n}}x\right) g\left( T_s^{-1}\left(\left(\frac{s}{r}\right)^{\frac{1}{n}} x\right)\right)dx\\
    &= \frac{r}{s}  \int_{\R^n} f(y) g(T_s^{-1}(x))\,dx\\
    &=\frac{r}{s}\cdot m_{f,g}(s).
\end{align*}
Therefore, we conclude that
\[
m_{f,g}(r) \leq m_{f,g}(s) \leq \frac{s}{r}\cdot m_{f,g}(r) \quad \text{ for all } 1 \leq r < s < \infty.
\]
This implies, in particular, that $m_{f,g}(r)$ is continuous on $[1,\infty)$. This proves (ii).

For the proof of (iii), according to (ii) we have 
\[
\int_{K_f}f(x) g_r(x)\,dx \to \int_{K_f}f(x) g(x) \,dx  \quad\text{ as } r \searrow 1. 
\]
Let $\{A_r\colon r>1\} \subset \text{GL}_n(\R)$ be such that $g\circ (A_rT_r)^{-1} = g$ holds for every $r >1$.  Therefore,
\[
\int_{\R^n} f(x) g(A_rx)\,dx \to \int_{\R^n}f(x)g(x)\,dx  \quad \text{ as } r\searrow 1. 
\]

\noindent Now as in the proof of Theorem \ref{t:exitence}, we have that $\{T_r\}_{r\geq 1}$ is contained in some compact set. Thus, $\{g_r\}_{r\geq 1}$ is uniformly bounded from above by some  log-concave function with finite maximum. By a functional version of Blaschke's selection theorem (see \cite[Theorem 2.15]{Mussnig-Li} and \cite[Theorems 4.18 and 7.6]{Rockafellar-Wets}), there exists a subsequence $g_{r_j}$ which hypo-converges to $g_0$ as $j\to\infty$. According to \cite[Theorem~7.7]{Rockafellar-Wets}, this condition is equivalent to local uniform convergence of $g_{r_j}$ to $g_0$. In particular, as $K_f$ is compact, this convergence is, in fact, uniform.
\end{proof}

\begin{proof}[Proof of Theorem~\ref{t:johnnew}] For brevity, set $f = 1_K \varphi$ and $g=1_{B_2^n}$, and let $r\searrow 1$.  For each such $r$, choose $T_r \in {\rm GL}_n(\R)$ with $\det(T_r) =r$ such that $m_{\mu}(r) = \int_{K} g(T_r^{-1}x)\,d\mu(x)$, which can be done because of Theorem~\ref{t:mipmeasures}(a). Set $g_r = g \circ T_r^{-1}$. According to Lemma~\ref{l:tech}, there exists a subsequence $\{r_j \colon j \in \mathbb{N}\}$ such that $fg_{r_j} \to fg$ uniformly. Set $d\mu(y) = f(y)g(y) \,d\sigma_{\mathbb{S}^{n-1}}(y)$. Choose a sequence $\{A_j \colon \det(A_j)=r_j, j \in \mathbb{N}\} \subset {\rm GL}_n(\R)$ such that $g = g_{r_j} \circ A_{j}^{-1}$.  Then we also have that $m_\mu(r_j) = r_j\int_{\R^n}f(A_jx) g(x)\,dx$ for each $j \in \N$. Therefore, for each $j \in \N$, the functions $f \circ A_j$ and $g$ maximize \eqref{e:2measureopt} with $r=r_j$.

By assumption, we have that $\vol_{n-1}(\partial K \cap T_r B_2^n) = 0$ for all $r$ sufficiently close to $1$; hence, there exists $j_0 \in \N$ such that for all $j \geq j_0$, we have  $\vol_{n-1}(\partial [A_j K] \cap \mathbb{S}^{n-1}) =0$. Hence, for all $j \geq j_0$, the hypotheses of Theorem~\ref{t:mipmeasures}(b) are satisfied for the pair $f \circ A_j$ and $g$. Consider the measures 
\[
 d\mu_j(y)=f(A_jy)g(y)\,d\sigma_{\mathbb{S}^{n-1}}(y), \quad j \geq j_0.
\]
Then for every $\theta \in \mathbb{S}^{n-1}$ and $j \geq j_0$, we have 
\begin{align*}
\int_{\mathbb{S}^{n-1}} \langle y, \theta \rangle ^2\,d\mu_j(y)=\frac{\mu_j(\mathbb{S}^{n-1})}{n}.
\end{align*}
Consequently, the measures given by 
\[
\nu_j(A) = \frac{\mu(A \cap [\partial K \setminus T_{r_j}B_2^n])}{\mu(\partial K \setminus T_{r_j}B_2^n)}
\]
must be isotropic whenever $j \geq j_0$.  Because $\mathbb{S}^{n-1}$ is a compact metric space, weak convergence implies that the measures $\nu_j$ ($j \geq j_0$) converge to some probability measure $\nu$ on $\mathbb{S}^{n-1}$. We will show that this measure $\nu$ is necessarily supported in $\partial K \cap \mathbb{S}^{n-1}$. 

Appealing to weak convergence and isotropicity, we see that, for each $\theta \in \mathbb{S}^{n-1}$,
\[
\int_{\mathbb{S}^{n-1}}\langle y,\theta\rangle^2\, d\nu_{j}(y) \to \int_{\mathbb{S}^{n-1}} \langle y,\theta\rangle^2\, d\nu(y),
\]
while simultaneously, 
\[
\frac{1}{n} = \frac{\nu_j(\mathbb{S}^{n-1})}{n} \to \frac{\nu(\mathbb{S}^{n-1})}{n}. 
\]
Consequently, the measure $\nu$ is isotropic. 

Finally, we show that $\text{supp}(\nu) \subset \partial K \cap \mathbb{S}^{n-1}$. Let $d(\cdot,\cdot)$ be the metric on $K$, and consider the set
\[
V_m= \left\{y \in \partial K \colon d(y, \mathbb{S}^{n-1}) \geq \frac{1}{m}\right\}.
\]
For each $j \geq j_0$, the measure $\nu_j$ is supported on $\partial K \setminus T_{r_j} B_2^n$, where $T_{r_j} B_2^n \to B_2^n$ as $j \to \infty$. Therefore, there exists some constant $N > 0$ such that, for any $m \geq N$, there is a $j(m)>0$ for which  $\nu_j(V_m) = 0$ whenever $j \geq j(m)$. As the set $V_m$ is open, weak convergences implies the inequality
\[
\nu(V_m) \leq \liminf_{j \to \infty} \nu_j(V_k) = 0,
\]
which implies $\nu(V_m) = 0$ for every $m \geq N$. Set 
\[
V = \bigcup_{m = N}^{\infty} V_m = \{x \in \partial K \colon d(x,\mathbb{S}^{n-1}) >0\} = \partial K \setminus \mathbb{S}^{n-1}.
\]
Then our previous investigation results in $\nu(V)=0$, whereby $\text{supp}(\nu) \subset \partial K \cap \mathbb{S}^{n-1}.$
\end{proof}

\section*{Acknowledgments}

The authors would like to thank Shiri Artstein-Avidan, J\'ulian Haddad, Carlos Hugo Jim\'enez, Eli Putterman, Rafael Villa and Artem Zvavitch for the inspiring discussions. 


\bibliographystyle{plain}
\bibliography{main}
	
	\vskip 2mm \noindent
 Steven Hoehner,   \ {\small \tt hoehnersd@longwood.edu} \\
	{\em 	Department of Mathematics \& Computer Science, Longwood University, U.S.A.}
  \vskip 1mm \noindent
	Michael Roysdon,   \ {\small \tt mar327@case.edu} \\
	{\em 	Department of Mathematics, Applied Mathematics and Statistics, Case Western Reserve University, U.S.A.}

\end{document}